\documentclass[12pt]{amsart}
\usepackage{amssymb,amsmath,mathrsfs}
\usepackage[colorlinks=true,urlcolor=blue,
citecolor=red,linkcolor=blue,linktocpage,pdfpagelabels,
bookmarksnumbered,bookmarksopen]{hyperref}
\usepackage[active]{srcltx}
\usepackage{verbatim}
\usepackage{epsfig,graphicx,color,mathrsfs}
\usepackage{graphicx}
\usepackage{amsmath,amssymb,amsthm,amsfonts}
\usepackage{amssymb}
\usepackage[english]{babel}
\usepackage[left=2.7cm,right=2.7cm,top=3cm,bottom=3cm]{geometry}

\newcommand{\R}{{\mathbb R}}

\def \d {\mathrm{d}}
\def \Cp {\mathrm{Cap}}
\def \ep {\varepsilon}

\def \de {\partial}
\def \e {\varepsilon}
\newtheorem{thm}{Theorem}[section]
\newtheorem{cor}[thm]{Corollary}
\newtheorem{lem}[thm]{Lemma}
\newtheorem{prop}[thm]{Proposition}

\newtheorem{defn}[thm]{Definition}

\numberwithin{equation}{section}

\title[Symmetry of singular solutions]{Symmetry and monotonicity of singular solutions to 
$p$-Laplacian systems involving a first order term}

\author[S.\,Biagi]{Stefano Biagi}
\author[F.\,Esposito]{Francesco Esposito}
\author[L.\,Montoro]{Luigi Montoro}
\author[E.\,Vecchi]{Eugenio Vecchi}

\address[S.\,Biagi]{Politecnico di Milano - Dipartimento di Matematica
\newline\indent
Via Bonardi 9, 20133 Milano, Italy}
\email{stefano.biagi@polimi.it}

\address[F.\,Esposito]{Dipartimento di Matematica e Informatica, Universit\`a della Calabria
\newline\indent
Ponte Pietro Bucci 31B, 87036 Arcavacata di Rende, Cosenza, Italy}
\email{francesco.esposito@unical.it}

\address[L.\,Montoro]{Dipartimento di Matematica e Informatica, Universit\`a della Calabria
\newline\indent
Ponte Pietro Bucci 31B, 87036 Arcavacata di Rende, Cosenza, Italy}
\email{montoro@mat.unical.it}

\address[E.\,Vecchi]{Dipartimento di Matematica, Universit\`a di Bologna
\newline\indent
Piazza di Porta San Donato 5, 40126 Bologna, Italy}
\email{eugenio.vecchi2@unibo.it}



\thanks{The authors are members of INdAM. 
F. Esposito and L. Montoro are partially supported by PRIN project 2017JPCAPN (Italy): 
{\em Qualitative and quantitative aspects of nonlinear PDEs.} L. Montoro is partially 
supported by  Agencia Estatal de Investigación (Spain): {\em project PDI2019-110712GB-100}.
}

\begin{document}
\begin{abstract}
We consider positive singular solutions (i.e.~with  a non-removable singularity) of a system of PDEs driven by $p$-La\-pla\-cian 
operators and with the 
additional presence of a nonlinear first order term. By a careful use of a rather new version of the moving 
plane method, we prove the symmetry of the solutions. The result is  already new in the scalar  case. 
\end{abstract}

\maketitle
%\tableofcontents

%%%%%%%%%%%%%%%%%%%%%%%%%%%%%%%%%%%%%%
\medskip

\section{Introduction}\label{intro}
The study of symmetry properties of solutions of PDEs is a long--standing topic which dates back to the 
foundational works of Alexandrov \cite{A}, Serrin \cite{serrin}, Gidas, Ni and Nirenberg \cite{GNN} and 
Berestycki and Nirenberg \cite{BN}. The common point among such papers is the use of the celebrated {\it 
moving plane method} in the case of semilinear elliptic equations. This powerful technique relies heavily 
on the validity of (weak and strong) maximum or comparison, principles. The relevance of the ideas which 
are at the core of the a\-fo\-re\-men\-tio\-ned papers is clearly witnessed by the huge number of 
contributions which are now available in the literature, whose aim has been to extend the results above mentioned to several different cases in the local framework: quasilinear equations 
(see, e.g., \cite{Da2, DP, DS1, DamSciunzi}) and 
cooperative elliptic systems  (see, e.g., \cite{busca, CMM, Dan, DeF, defig2,Troy}).\\
A slightly different line of research concerns the study of symmetry properties of {\em singular} solutions. 
In this case, we refer, e.g., to \cite{CLN2, Ter} for the case of point-singularity.
More re\-cen\-tly, Sciunzi \cite{Dino} developed a new method which allows to study symmetry properties of 
solutions which are singular on sets of {\it small capacity}. \\
The aim of this paper is to  focus on systems of PDEs 
driven by the $p$-Laplacian with the additional presence of a nonlinear gradient term. To be more precise,
 we will consider solutions $u_i \in C^{1,\alpha}(\overline{\Omega}\setminus \Gamma)$ ($i=1,\ldots, m$) to the 
 following system of PDEs
 \begin{equation}\tag{$\mathcal S$}\label{eq:System}
 \left\{\begin{array}{rl}
 -\Delta_{p_i}u_i + a_i(u_i)|\nabla u_i|^{q_i} = 
 f_i(u_1, \ldots, u_m) & \textrm{in } \Omega \setminus \Gamma\\
 u_i >0 & \textrm{in } \Omega\setminus \Gamma\\
 u_i = 0 & \textrm{on } \partial \Omega.
\end{array}\right.
\end{equation}
Here, $\Omega \subseteq \mathbb{R}^N$ is a smooth and bounded domain and for suitable smooth functions 
\begin{equation*}
-\Delta_{p}v := \mathrm{div}\left(|\nabla v|^{p-2}\nabla v\right)
\end{equation*}
\noindent denotes the $p$-Laplacian. Generally, solutions to equations involving the 
$p$-Laplace operator are  of class $C^{1,\alpha}$. This assumption 
is natural according to classical 
regularity results \cite{DB,Li,T}. We  point out that,  in this paper, we deal with  singular solutions  that are $C^{1,\alpha}$ far from the critical set $\Gamma$. We consider  solutions    with a non-removable singularity:  we mean that solutions possibly do not admit a smooth extension in $\Omega$, i.e. it is not possible to find $\tilde u_i\in W^{1,p}_0(\Omega)$ such that $u_i\equiv \tilde u_i$ in $\Omega\setminus \Gamma$. Indeed,  without any a priori assumption, the gradient of the solutions possibly blows up near the 
critical set and hence each equations of \eqref{eq:System} may exhibit both a degenerate  and a 
singular 
nature at the same time.
\\ 
 Before stating our main result,
 we first introduce
 the main \emph{structural as\-sump\-tions}
 we require on problem \eqref{eq:System}; these assumptions
 will be tacitly understood in the sequel.
 \medskip
 
 \noindent\textbf{Assumptions.} Throughout what follows, we assume that
 \begin{itemize}
  \item[$({h}_\Omega):$] $\Omega$ is a bounded domain in $\R^N$ (with $N\geq 2$), which is 
  \emph{convex} in the $x_1$-di\-rec\-tion and symmetric with respect to the hyperplane
  $$\Pi_0 := \{x\in\R^N:\,x_1 = 0\};$$
  \item[$({h}_\Gamma):$] $\Gamma\subseteq\Omega\cap\Pi_0$ is a compact set 
  satisfying
  $\mathrm{Cap}_{p}(\Gamma) = 0$, where
  $$p := \max_{1\leq k\leq m}p_k;$$
  \item[$({h}_{p,q}):$] for every $i = 1,\ldots,m$, we have
  \begin{equation}\label{eq:arnoldstrong}
  \begin{split}
  \frac{2N+2}{N+2} < p_i\leq N\qquad&\text{and}
  \qquad \max\{p_i-1,1\}\leq q_i< p_i\\
  &\,\,\text{or}
  \\
2\leq  p_i\leq N\qquad&\text{and}
  \qquad q_i =p_i.
  \end{split}
  \end{equation}
  \item[$({h}_{a}):$] 
  the functions $a_i(\cdot)$ are \emph{locally Lipschitz-continuous}
  on $\mathcal{I} := [0,+\infty)$, that is,
  for e\-ve\-ry $M > 0$ there exists a constant 
  $L = L_{M} > 0$ such that
  $$|a_i(t)-a_i(s)|\leq L|t-s|\qquad \text{for every $t,s\in[0,M]$};$$
  \item[$(h_f):$] 
  the functions $f_i(\cdot)$ are \emph{of class $C^1$}
  on $\mathcal{I}^m$; moreover, we assume that
  \begin{itemize}
   \item[(a)] $f_i(t_1,\ldots,t_m) > 0$  for every $t_1,\ldots,t_m > 0$;
   \vspace*{0.05cm}
   
   \item[(b)] $\de_{t_k}f_i(\cdot)\geq 0$ for every $1\leq i,k\leq m$ with $i\neq k$.
  \end{itemize}
 \end{itemize}
We remark that, without loss of generality, we are  considering the $x_1$-direction  being the operator  invariant with respect to translations and rotations.

 Then, we properly define what we mean by a \emph{weak solution}
 to problem \eqref{eq:System}.
 \begin{defn} \label{def:weaksol}
 We say that a vector-valued function
 $$\mathbf{u} = (u_1,\ldots,u_m)\in C^{1,\alpha}(\overline{\Omega}\setminus\Gamma;\R^m)$$
 is a \emph{weak solution} to \eqref{eq:System} 
 if it satisfies the properties listed below:
  \begin{enumerate}
   \item for every $i = 1,\ldots,m$ and every $\varphi\in C_c^1(\Omega\setminus\Gamma)$, one has
   \begin{equation} \label{eq:weaksoldef}
    \int_{\Omega}|\nabla u_i|^{p_i-2}\langle \nabla u_i,\nabla\varphi\rangle\,d x
    + \int_{\Omega}a_i(u_i)|\nabla u_i|^{q_i}\varphi\,d x
    = \int_{\Omega}f_i(u_1,\ldots,u_m)\varphi\,d x;
   \end{equation}
   \item 
   $u_i > 0$ pointwise in $\Omega\setminus\Gamma$ \emph{(}for every $i = 1,\ldots,m$\emph{)};
   \item $u_i = 0$ on $\de \Omega$ 
   \emph{(}for every $i = 1,\ldots,m$\emph{)}.
  \end{enumerate}
\end{defn}
Our main results is the following
\begin{thm}\label{thm:symmetry}
 Let assumptions $(h_\Omega)$-to-$(h_f)$ be in force, and let
 $\mathbf{u} \in C^{1,\alpha}(\overline{\Omega}\setminus \Gamma; \mathbb{R}^m)$ 
 be a weak solution to problem \eqref{eq:System}. Then, the following facts hold:
 \begin{itemize}
 	\item[(i)] $\mathbf{u}$ is symmetric with respect to the hyperplane $\Pi_0$, namely
 	$$\mathbf{u}(x_1,x_2, \dots, x_N) = \mathbf{u}(-x_1,x_2, \dots, x_N) \quad \text{in } \Omega.$$
 	\item[(ii)] $\mathbf{u}$ is non-decreasing in the $x_1$-direction in the set $\Omega_0 =\{x_1<0\}$ and moreover 
\begin{equation}\label{eq:ehvabe}
\partial_{x_1} u_i > 0 \quad \text{in } \Omega_0,
\end{equation}
 	for every $i \in \{1,\dots, m\}$.
 \end{itemize}
\end{thm}
We remark that, although the technique that we will develop to prove Theorem \ref{thm:symmetry} works for any $p_i> N$, the result is stated in the range \eqref{eq:arnoldstrong} since there are no sets (different from the empty-set) of zero $p_i$-capacity when $p_i>N$.

We now want to spend a few comments about 
Theorem \ref{thm:symmetry}.
Firstly, we notice that Theorem \ref{thm:symmetry} extends the results in \cite{EMM} to the case of 
singular (non regular) solutions. 
The main novelty of our result concerns the singular case $p_i<2$ (for some $i=1,\ldots,m$). 
In order to 
explain this,
let us consider the case of a scalar equation with $a(u)\equiv 0$: our problem boils then down to the case 
considered in \cite{EMS} (which is therefore our benchmark case), that is, 
$$\left\{\begin{array}{rl}
 -\Delta_{p}u  = 
 f(u) & \textrm{in } \Omega \setminus \Gamma\\
 u >0 & \textrm{in } \Omega\setminus \Gamma\\
 u = 0 & \textrm{on } \partial \Omega.
\end{array}\right.$$
In \cite{EMS} the Authors, already in  the scalar  case, require some a priori assumptions on the nonlinearity $f$ involved in the problem, while here we are able to remove  the growth a priori assumption made in \cite{EMS} on the nonlinearities $f_i$. This aspect 
is strictly related with a technical issue that has to be faced when using the {\it integral version} of 
the moving plane method. In particular, one needs to apply a weighted Sobolev inequality due to Trudinger 
\cite{Tru}, whose validity depends on proper summability conditions of the weight, which happens to be of 
the form $|\nabla u|^{p-2}$. In \cite{EMS} this condition is proved to be satisfied after a nice study of 
the behavior of the gradient of the solution near the singular set $\Gamma$, 
which is based on a subtle 
growth estimate proved in \cite{PQS}, thanks to a priori assumption on the nonlinearity $f$ involved in the problem. Since we are lacking an analogous estimate for our solutions (even in 
the simpler case $a_i \equiv 0$ for every $i=1,\ldots,m$), we are not even entitled to profit of the 
weighted Sobolev inequality of Trudinger. Nevertheless, we 
can avoid such a priori assumptions by exploiting some a priori regularity results that we  prove in this paper, see Lemma \ref{leaiuto} and Lemma \ref{lem:SobolevSplitting}  below. 
We stress that this is a considerable simplification 
of the proof in \cite{EMS} and it is 
crucially based on the new Lemma \ref{lem:SobolevSplitting}.

Now, by scrutinizing the proof of Theorem \ref{thm:symmetry}, one can easily
recognize that a key ingredient for our argument is the fact that,
under assumption $({h}_{p,q})$, we have 
$$p_i^* = \frac{Np_i}{N-p_i}\geq 2\qquad(\text{for all $i = 1,\ldots,m$)}.$$
However, since the above inequality is equivalent to require that
$p_i \geq 2N/(N+2)$,
it is natural to wonder why 
we \emph{require} $p_i$ to satisfy the worse lower bound
$$p_i > \frac{2N+2}{N+2} =: \beta_N.$$
The main reason behind this choice is that, when $p \leq \beta_N$, we 
\emph{do not have}
at our disposal a Strong Comparison Principle
for the operator $-\Delta_pu+a(u)|\nabla u|^q$ which 
allows us to do not take care of the \emph{critical set} $\mathcal{Z} = \{\nabla u = 0\}$ that is, in general,  a crucial point when one deals with the $p$-Laplace operator.

In view of this fact, in order 
to establish Theorem~\ref{thm:symmetry} when
$$\frac{2N}{N+2}\leq p_i \leq \frac{2N+2}{N+2}$$
we would need to recover a technical lemma analogous to \cite[Lemma 3.2]{EMS}, which also requires 
suitable estimates for the second-order derivatives of the solution
of \eqref{eq:System}. Even if these estimates are available in our setting (see 
\cite[Theorem 2.1]{EMM}), we prefer
to avoid this technicality here: indeed, we believe that the main novelties of our work
are both to consider systems involving general first order terms  and the possibility of considering the case $p_i < 2$ without the need of prescribing  precise
 (a priori) growth estimates for the solution of \eqref{eq:System}. To the best of our knowledge, our result is new also in the case of a single scalar equation.  

\medskip
 The paper is organized as follows: we prove some technical results in Section \ref{sec.preliminaries} that we will exploit in
Section  \ref{sec:proofThm} to prove  Theorem \ref{thm:symmetry}.

\section{Notation and preliminary results} \label{sec.preliminaries}
 In this first section we collect all the relevant notation
 which shall be used throughout the paper; moreover, we review some
 results \emph{already known} in the literature which shall be fundamental for the 
 proof of our Theorem \ref{thm:symmetry}. We will adopt the symbol $|A|$ to denote the $N$-dimensional 
 Lebesgue measure of a measurable set $A \subseteq \mathbb{R}^N$.
 
 \subsection{Capacity.} 
 Due to the important role played by assumption $(h_\Gamma)$ in our ar\-gu\-ments, we briefly
 recall a few notions and results concerning the \emph{Sobolev capacity}.
 \vspace*{0.1cm}
 
 Let $K\subseteq\R^N$ be a \emph{compact set},
 and let $\mathcal{O}\subseteq\R^N$ be an open set such that $K\subseteq\mathcal{O}$.
 Given any $1 < r \leq N$, the $r$-capacity 
 of the \emph{condenser} $\mathcal{C} := (K,\mathcal{O})$
 is defined as
 $$\Cp_r(K,\mathcal{O}) 
 := \inf \bigg\{\int_{\mathbb{R}^N}|\nabla \phi|^r\,\d x\,:\,
  \text{$\phi \in C^{\infty}_c(\mathcal{O})$ and $\phi\geq 1$ on $K$}\bigg\}.$$
 We then say that $K$ has \emph{vanishing $r$-capacity}, and we write
 $\Cp_r(K) = 0$, if
 $$\Cp_r(K,\mathcal{O}) = 0\quad\text{\emph{for every} open set $\mathcal{O}\supset K$}.$$
 As is reasonable to expect, compact sets with vanishing $r$-capacity
 have to be \emph{very small}; the next theorem shows that
 this is actually true in the sense of Hausdorff measure.
 \begin{thm} \label{thm.propCp}
 Let $K\subseteq\R^N$ be a compact set. Then, the following assertions hold.
 \begin{enumerate}
   \item If $\Cp_r(K) = 0$, then $\mathcal{H}^s(K) = 0$ for every $s > N - r$.
   \vspace*{0.05cm}
   
   \item If $r < N$ and $\mathcal{H}^{N-r}(K) < \infty$, then $\Cp_r(K) = 0$.
\end{enumerate}  
\end{thm}
 For a proof of Theorem \ref{thm.propCp}, 
 we refer the Reader to \cite[Sec.\,2.24]{HeinBook}.
 \begin{cor} \label{cor.pqCap}
  Let $1< r_1 < r_2 \leq N$ and let $K\subseteq\R^N$ be compact.
  Then,
  $$\Cp_{r_2}(K) = 0\,\,\Longrightarrow\,\,\Cp_{r_1}(K) = 0.$$
 \end{cor}
 \begin{proof}
  Since, by assumption, $\Cp_{r_2}(K) = 0$, by Theorem \ref{thm.propCp}-(1)
  we have $\mathcal{H}^s(K) = 0$ for every $s > N-r_2$;
  in particular, since $r_1 < r_2$, we derive that
  $$r_1 < N\qquad\text{and}\qquad \mathcal{H}^{N-r_1}(K) = 0.$$
  Using this fact and Theorem \ref{thm.propCp}-(2), we then
  conclude that $\Cp_{r_1}(K) = 0$.
 \end{proof}
 On account of Corollary \ref{cor.pqCap}, if $\Gamma\subseteq\Omega\cap\Pi_0$
 is as in Theorem \ref{thm:symmetry}, we have
 \begin{equation} \label{eq:GammaCpizero}
  \Cp_{p_i}(\Gamma) = 0\qquad\text{for every $i = 1,\ldots,m$}.
 \end{equation} 
 \noindent Let now $K\subseteq\R^N$ be a fixed compact set \emph{with vanishing $r$-capacity}, and let
 $\mathcal{O}\subseteq\R^N$ be an open set containing $K$. We aim to show  
 that it is possible to construct a family
 of functions in $\R^N$, say $\{\psi_\e\}_\e$, 
 which satisfies the following properties:
 \begin{enumerate}
  \item $\psi_\varepsilon\in \mathrm{Lip}(\R^N)$ and $0\leq \psi_\e\leq 1$ pointwise in $\R^N$;
  \item there exists
  an open neighborhood $\mathcal{V}_\e\subseteq\mathcal{O}$ of $K$ such that
  $$\text{$\psi_\e\equiv 0$ on $\mathcal{V}_\e$};
  $$
  \item $\psi_\e(x)\to 1$ as $\e\to 0^+$ for a.e.\,$x\in\R^N$;
  \item there exists a constant $C_0 > 0$, only depending on $r$, such that
  $$\int_{\R^N}|\nabla \psi_\e|^r\,d x \leq C_0\varepsilon.$$
 \end{enumerate}
 In fact, let $\varepsilon_0 = \e_0(K,\mathcal{O}) > 0$ be such that
 $$\mathcal{B}_\varepsilon := \big\{x\in\R^N:\,\mathrm{dist}(x,K) < \varepsilon\big\}
 \Subset \mathcal{O}\qquad\text{for every $0<\e<\e_0$}.$$
 Since $\Cp_r(K) = 0$ and $\mathcal{B}_\e\subseteq\R^N$ is an open neighborhood
 of $K$, for every $\e\in (0,\e_0)$
 there exists a smooth function $\phi_\e\in C_c^\infty(\mathcal{B}_\e)$ such that
 \begin{equation} \label{eq:propphie}
 \text{$\phi_\e \geq 1$ on $K$}\quad \text{and}\quad
 \int_{\R^N}|\nabla \phi_\e|^r\,d x < \varepsilon.
 \end{equation}
  We then consider the Lipschitz functions
 \begin{itemize}
  \item[(i)] $T(s) := \max\{0;\min\{s;1\}\}$ (for $s\in\R$), \vspace*{0.1cm}
  \item[(ii)] $g(t) := \max\{0;1-2t\}$ (for $t\geq 0$)
 \end{itemize}
 and we define, for every $\varepsilon\in (0,\varepsilon_0)$,
 \begin{equation} \label{test1}
  \psi_\e:\R^N\to\R,\qquad 
  \psi_{\varepsilon}(x):= g\big(T(\varphi_{\varepsilon}(x))\big).
 \end{equation}
 Clearly, $\psi_\e\in\mathrm{Lip}(\R^N)$ and $0\leq\psi_\e\leq 1$ in $\R^N$; moreover,
 by \eqref{eq:propphie} we have
 $$\int_{\R^N}|\nabla\psi_\e|^r\,d x\leq 2^r\int_{\R^N}|\nabla\phi_\e|^r\,d x
 \leq 2^r\e.$$
 Hence,  the family $\{\psi_\e\}_\e$ satisfies
 properties (1) and (4) (with $C_0 = 2^r$).
 As for the validity of property (2) we observe that,
 by the explicit definitions of $T$ and $g$, we have
 $$\psi_\e\equiv 0\quad\text{on $\mathcal{V}_\e := \big\{\phi_\e > 1/2\big\}$};$$
 as a consequence, since $\mathcal{V}_\e$ is an open neighborhood of $K$ and
 since $\mathcal{V}_\e\subseteq\mathcal{B}_\e\subseteq \mathcal{O}$ (remind that
 $\phi_\e\geq 1$ on $K$ and $\mathrm{supp}(\phi_\e)\subseteq\mathcal{B}_\e$), 
 we immediately conclude that also property (2) is sa\-ti\-sfied.
 Finally, the validity of property (3) follows from the fact that
 \begin{equation} \label{eq:psiequiv1}
  \psi_\e \equiv g(0) = 1\quad\text{on $\R^N\setminus\mathcal{B}_\e$}.
  \end{equation}
 Indeed, since the family $\{\mathcal{B}_\e\}_{\e}$ shrinks to $K$ as $\e\to 0^+$, by
 \eqref{eq:psiequiv1} we get
 $$\text{$\psi_\e(x)\to 1$ as $\e\to 0^+$ \quad \emph{for every $x\in\R^N\setminus K$}};$$
 from this, since $|K| = \mathcal{H}^N(K) = 0$ (remind that $\Cp_p(K) = 0$ and
 see Theorem \ref{thm.propCp}), we immediately conclude 
 that the family $\{\psi_\e\}_\e$ satisfies also property (3), as claimed.
 \medskip
 
 Throughout what follows, we will repeatedly use the family $\{\psi_\e\}_\e$
 with dif\-fe\-rent choices of $K$ and $\mathcal{O}$; hence, to simplify the notation,
 we shall refer to this family as a \emph{cut-off family for the compact set
 $K$, related with the open set $\mathcal{O}$}.
 \subsection{Notations for the moving plane method.}
 Let
 $\mathbf{u} 
 \in C^{1,\alpha}(\overline{\Omega}\setminus\Gamma;\R^m)$ be a weak so\-lu\-tion
 to problem \eqref{eq:System}.
 For every fixed $\lambda\in\R$, we denote by $R_\lambda$
 the reflection trough the hy\-per\-pla\-ne $\Pi_\lambda := \{x\in\R^N:\,x_1 = \lambda\}$,
 that is,
 \begin{equation} \label{eq.defRlambda}
  R_\lambda(x) = x_\lambda := (2\lambda-x_1,x_2,\ldots,x_N);
 \end{equation}
 accordingly, we introduce the vector-valued function
 \begin{equation} \label{eq.defulambda}
  \mathbf{u}_\lambda(x) =
  (u_{1,\lambda}(x),\ldots,u_{m,\lambda}(x)) := \mathbf{u}(x_\lambda), \qquad
  \text{for $x\in R_\lambda(\overline{\Omega}\setminus\Gamma)$}.
 \end{equation}
 We point out that, since $\mathbf{u}$ solves \eqref{eq:System}, it is easy to see that
 \begin{enumerate}
  \item  
 $\mathbf{u}_\lambda\in C^{1,\alpha}(R_\lambda(\overline{\Omega}\setminus\Gamma);\R^m)$;
  \item for every $i = 1,\ldots,m$ and
  every $\varphi\in C^1_c(R_\lambda(\Omega\setminus\Gamma))$,
  one has
  \begin{equation} \label{eq.PDEulambda}
  \int_{R_\lambda(\Omega)}|\nabla u_{i,\lambda}|^{p_i-2}
  \langle \nabla u_{i,\lambda},\nabla \varphi\rangle\,d x
  + \int_{R_\lambda(\Omega)}a_i(u_{i,\lambda})|\nabla u_{i,\lambda}|^{q_i}\varphi\,d x
  = \int_{R_\lambda(\Omega)}f_i(\mathbf{u}_\lambda)\varphi\,\d x;
  \end{equation}
    \item 
  $u_{i,\lambda} > 0$ pointwise $R_\lambda(\Omega\setminus\Gamma)$ (for every $i = 1,\ldots,m$);
  \item $u_{i,\lambda} \equiv 0$ on $R_\lambda(\partial\Omega)$
  (for every $i = 1,\ldots,m$). 
 \end{enumerate}
 To proceed further, we let
 \begin{equation} \label{eq.defaOmega}
  \varrho = \varrho_\Omega := \inf_{x\in\Omega}x_1
 \end{equation}
 and we observe that, since $\Omega$ is 
 bounded and symmetric with respect to
 $\Pi_0$,
 we certainly have $-\infty < \varrho < 0$.
 As a consequence, for every $\lambda\in(\varrho,0)$ we can set
 \begin{equation} \label{eq.defOmegalambda}
 \Omega_\lambda := \{x\in\Omega:\,x_1<\lambda\}.
 \end{equation}
 We explicitly point out that, since
 $\Omega$ is convex in the $x_1$-direction, we have
 \begin{equation} \label{eq.inclusionOmegalambda}
  \Omega_\lambda\subseteq R_\lambda(\Omega)\cap \Omega.
 \end{equation}
 Finally, for every $\lambda\in(\varrho,0)$ we define the function
 $$\mathbf{w}_\lambda(x) = 
 (w_{1,\lambda}(x),\ldots,w_{m,\lambda}(x)) :=
 (\mathbf{u}-\mathbf{u}_\lambda)(x), \qquad \text{for
 $x\in (\overline{\Omega}\setminus\Gamma)\cap R_\lambda(\overline{\Omega}\setminus\Gamma)$}.$$
 On account of \eqref{eq.inclusionOmegalambda}, 
 $\mathbf{w}_\lambda$ is surely well-posed on
 $\overline{\Omega}_\lambda\setminus R_\lambda(\Gamma)$.
 
 \subsection{Preliminary results.}
 After these preliminaries, we 
 devote the remaining part of this section
 to collect some auxiliary results which shall
 be useful for the proof of Theorem \ref{thm:symmetry}.
 In what follows, we tacitly inherit all the notation
 introduced so far.
 \medskip 
 
 To begin with, we recall some
 identities between vectors in $\R^N$ which
 are useful in dealing
 with quasilinear operators:
 \emph{for every $s> 1$ 
 there exist constants $C_1,\ldots,C_4 > 0$, only depending on $s$, such that,
 for every $\eta,\eta'\in\R^N$, one has}
 \begin{equation}\label{eq:inequalities}
  \begin{split}
 & \langle |\eta|^{s-2}\eta-|\eta'|^{s-2}\eta', \eta- \eta'\rangle \geq C_1
(|\eta|+|\eta'|)^{s-2}|\eta-\eta'|^2, \\[0,15cm]
& \big| |\eta|^{s-2}\eta-|\eta'|^{s-2}\eta '| \leq C_2
(|\eta|+|\eta'|)^{s-2}|\eta-\eta '|,\\[0.15cm]
& \langle |\eta|^{s-2}\eta-|\eta'|^{s-2}\eta ', \eta-\eta ' \rangle \geq C_3
|\eta-\eta '|^s \qquad (\text{if $s\geq 2$}), \\[0.15cm]
& \big| |\eta|^{s-2}\eta-|\eta'|^{s-2}\eta '| \leq C_4|\eta-\eta'|^{s-1}
\qquad (\text{if $1<s<2$}).
\end{split}
\end{equation}
We refer, e.g., to \cite{Da2} for a proof of 
\eqref{eq:inequalities}. \medskip

 We then establish the following fundamental lemma.
 \begin{lem}\label{leaiuto} Let assumptions $(h_\Omega)$-to-$(h_f)$ be in force. 
 Let 
   $i \in\{1,\ldots,m\}$ be fixed, and let $\lambda\in(\varrho,0)$ be such that
   $$R_\lambda(\Gamma)\cap\overline{\Omega}_\lambda\neq \varnothing.$$
   Then, there exists a constant $\mathbf{c} = \mathbf{c}_i > 0$ such that
	\begin{equation} \label{eq.lemmaSumm}
    \int_{\Omega_\lambda} \left(|\nabla u_i| + 
    |\nabla u_{i,\lambda}|\right)^{p_i -2} |\nabla w_{i,\lambda}^{+}|^2\,d x  \leq
	\mathbf{c}_i.
	\end{equation}
 \end{lem}
 \begin{proof}
 We first notice that, since $\lambda < 0$, by assumption $(h_\Gamma)$ 
 we have $ \overline{\Omega}_\lambda\cap \Gamma = \varnothing$.
 Thus, the function $\mathbf{u}
 = (u_1,\ldots,u_m)$ is of class $C^{1,\alpha}$
 on $\overline{\Omega}_\lambda$, and we can set
 $$M = M_{\mathbf{u}} := \max_{1\leq j\leq m}
 \big(\|u_j\|_{L^\infty(\overline{\Omega}_\lambda)}+
 \|\nabla u_j\|_{L^\infty(\overline{\Omega}_\lambda)}\big)
 < +\infty.$$
 Moreover, since 
 $u_i$ and $u_{i,\lambda}$ are non-negative, we have
 \begin{equation} \label{eq:estimwplusM}
  0\leq w_{i,\lambda}^+ = (u_i-u_{i,\lambda})^+
 \leq u_i\leq M\qquad\text{pointwise
 in $\overline{\Omega}_\lambda\setminus R_\lambda(\Gamma)$}.
 \end{equation}
 Now, to prove \eqref{eq.lemmaSumm} we distinguish two cases.
  \begin{itemize}
   \item[(i)] $\max\{1, p_i-1\} \leq q_i < p_i$;
   \vspace*{0.1cm}
   \item[(ii)] $q_i = p_i$.
  \end{itemize}
  \vspace*{0.1cm}
  
  \noindent \textbf{Case (i)}. First of all we observe that,
  since $R_\lambda$ is a bijective
  linear map and since $\Gamma$ has va\-ni\-shing $p_i$\--ca\-pa\-city
  (see \eqref{eq:GammaCpizero}), the compact set $\Gamma_\lambda := R_\lambda(\Gamma)$ satisfies
  $$\Cp_{p_i}(\Gamma_\lambda) = 0.$$
  As a consequence, if $\mathcal{O}_\lambda\subseteq\R^N$ is a fixed open neighborhood
  of $\Gamma_\lambda$, we can choose a \emph{cut-off family} $\{\psi_\e\}_{\e<\e_0}$
  for $\Gamma_\lambda$, related with the open set $\mathcal{O}_\lambda$.
  This means, precisely, that
  \begin{enumerate}
  \item $\psi_\varepsilon\in \mathrm{Lip}(\R^N)$ and $0\leq \psi_\e\leq 1$ pointwise in $\R^N$;
  \item there exists
  an open neighborhood $\mathcal{V}^{\lambda}_\e\subseteq\mathcal{O}_\lambda$ of $\Gamma_\lambda$ 
  such that
  $$\text{$\psi_\e\equiv 0$ on $\mathcal{V}^{\lambda}_\e$};
  $$
  \item $\psi_\e(x)\to 1$ as $\e\to 0^+$ for a.e.\,$x\in\R^N$;
  \item there exists a constant $C_0 > 0$, independent of $\e$, such that
  $$\int_{\R^N}|\nabla \psi_\e|^{p_i}\,d x \leq C_0\varepsilon.$$
 \end{enumerate}
  We now define, for every $\varepsilon \in (0,\e_0)$,
  the map
  $$\varphi_{i,\varepsilon}(x):= 
  \begin{cases}
  	w_{i,\lambda}^+(x)\,\psi_\varepsilon^{p_i}(x) = (u_i-u_{i,\lambda})^+(x)\,\psi_\varepsilon^{p_i}(x),
  	& \text{if $x\in\Omega_\lambda$}, \\
  	0, & \text{otherwise}.
  \end{cases}$$
We then claim that the following assertion hold:
\begin{itemize}
	\item[(i)] $\varphi_{i,\varepsilon} \in \mathrm{Lip}(\R^N);$
	
	\item[(ii)] $\operatorname{supp}(\varphi_{i, \varepsilon}) \subseteq \Omega_\lambda$ and $\varphi_{i, \varepsilon} \equiv 0$ near $\Gamma_\lambda$.
\end{itemize}
In fact, since $u_i \in C^{1,\alpha}(\overline{\Omega}_\lambda)$ and 
$u_{i,\lambda} \in C^{1,\alpha}(\overline{\Omega}_\lambda \setminus \Gamma_\lambda)$, we readily see that 
  $$\text{$w_{i,\lambda}^+\in \mathrm{Lip}(\overline{\Omega}_\lambda\setminus V)$
  \emph{for every o\-pen set $V\supseteq \Gamma_\lambda$}};$$
  from this, since
  $\psi_\varepsilon\in\mathrm{Lip}(\R^N)$
  and $\psi_\varepsilon\equiv 0$ on $\mathcal{V}^\lambda_\e\supseteq\Gamma_\lambda$,
  we get
  $\varphi_{i,\varepsilon} \in \mathrm{Lip}(\overline{\Omega}_\lambda)$.
  On the other hand, since 
  $\varphi_{i,\varepsilon}\equiv 0$ on $\partial\Omega_\lambda$, we easily conclude that
  $\varphi_{i,\ep}\in \mathrm{Lip}(\R^N)$,
  as claimed. As for assertion (ii), it follows from
  the definition of $\varphi_{i,\ep}$, jointly with the fact that
  $$\text{$\psi_\ep\equiv 0$ on $\mathcal{V}^\lambda_\ep\supseteq \Gamma_\lambda$}.$$
  On account of properties (i)-(ii) of $\varphi_{i,\ep}$,
   a standard density argument allows us to use
  $\varphi_{i,\ep}$ as a test function \emph{both in \eqref{eq:weaksoldef} and
  \eqref{eq.PDEulambda}}, obtaining
  \begin{equation}\label{eq:weakforsub}
	\begin{split}
	&
     \int_{\Omega_\lambda}
     \langle |\nabla u_i|^{p_i-2} \nabla u_i - |\nabla u_{i,\lambda}|^{p_i-2}\nabla u_{i,\lambda},\nabla
     \varphi_{i,\ep}\rangle\, dx \\
     &
     \qquad\qquad + \int_{\Omega_{\lambda}}\big(a_i(u_i)|\nabla u_i|^{q_i} - a_i(u_{i,\lambda})
     |\nabla u_{i,\lambda}|^{q_i}\big)\varphi_{i,\ep}\, dx \\
     & \quad
     = \int_{\Omega_\lambda}(f_i(\mathbf{u})-f_i(\mathbf{u}_\lambda))
     \varphi_{i,\ep}\, dx.
     \end{split}
  \end{equation}
  By unraveling the very definition of $\varphi_{i,\ep}$, we then obtain
  \begin{equation} \label{eq.identityIntermediate}
	\begin{split} 
     &
    \int_{\Omega_\lambda}
     \langle |\nabla u_i|^{p_i-2} \nabla u_i - 
     |\nabla u_{i,\lambda}|^{p_i-2}\nabla u_{i,\lambda},\nabla
     w_{i,\lambda}^{+}\rangle \psi_{\ep}^{p_i}\, dx
     \\
      &\qquad\qquad +p_i  \int_{\Omega_\lambda}
     \langle |\nabla u_i|^{p_i-2} \nabla u_i - |\nabla u_{i,\lambda}|^{p_i-2}\nabla u_{i,\lambda},\nabla
     \psi_{\ep}\rangle w_{i,\lambda}^{+} \psi_{\ep}^{p_i -1}\, dx \\
     & \qquad\qquad 
    + \int_{\Omega_{\lambda}}\big( a_i(u_i)|\nabla u_i|^{q_i} - 
    a_i(u_{i,\lambda})|\nabla u_{i,\lambda}|^{q_i}\big) w_{i,\lambda}^{+}\psi_{\ep}^{p_i}\, dx \\
    & \quad
    = \int_{\Omega_\lambda} (f_i(\mathbf{u})-f_i(\mathbf{u}_\lambda))
    w_{i,\lambda}^{+}\psi_{\ep}^{p_i}\,d x.
	\end{split}
  \end{equation}
  We now observe that the integral in the
  right-hand side of \eqref{eq.identityIntermediate} is actually
  performed on the set
  $A_{i,\lambda} := \{x\in\Omega_\lambda:\,u_i\geq u_{i,\lambda}\}\setminus \Gamma_\lambda$;
  moreover, we have
  \begin{equation} \label{eq:estimAilambda}
   0\leq u_{i,\lambda}(x)\leq u_i(x)
  \leq M\quad
  \text{for every $x\in A_{i,\lambda}$}.
  \end{equation}
The right hand side of \eqref{eq.identityIntermediate} can be arranged  as follows
\begin{equation} \label{lipf}
	\begin{split}
		&\int_{\Omega_\lambda}(f_i(\mathbf{u})-f_i(\mathbf{u}_\lambda))
		\,w_{i,\lambda}^+\,\psi_\ep^{p_i}\, dx = 
		\int_{\Omega_\lambda}[f_i(u_1,\ldots,u_m) - 
		f_i(u_{1, \lambda},\ldots,u_{m,\lambda})] \, w_{i,\lambda}^+\,\psi_\ep^{p_i}\, dx\\
		& \quad =\int_{\Omega_\lambda}\big[f_i(u_1,u_2,\ldots,u_m)-f_i(u_{1,\lambda},u_2,\ldots,u_m) 		
		  \\
		  & \quad\qquad\quad +f_i(u_{1,\lambda},u_2,\ldots,u_m) -
		  f_i(u_{1,\lambda},u_{2,\lambda},\ldots,u_{m,\lambda})\big]
		  \,w_{i,\lambda}^+\,\psi_\ep^{p_i}\, dx\\
	& \quad = \int_{\Omega_\lambda}
	\big[f_i(u_1, u_2,\ldots,u_m)-f_i(u_{1,\lambda},u_2,\ldots,u_m) \\
	&\quad\qquad\quad 
	+f_i(u_{1,\lambda},u_2,\ldots,u_m)-f_i(u_{1,\lambda},u_{2,\lambda} \ldots,u_m) \\
	& \quad\qquad\quad\,\,
	\vdots \\
	& \quad \qquad\quad
	+ f_i(u_{1,\lambda},u_{2,\lambda}, \ldots, u_i, \ldots,u_m)
	-f_i(u_{1,\lambda},u_{2,\lambda}\ldots, u_{i,\lambda}, \ldots,u_m)   \\
	& \quad \qquad\quad\,\,\vdots\\
    & \quad \qquad\quad
    +f_i(u_{1,\lambda},u_{2,\lambda}, u_{3, \lambda} 
    \ldots,{u}_m)-f_i(u_{1,\lambda},u_{2,\lambda}, u_{3,\lambda}\ldots,u_{m,\lambda})\big]
    \,w_{i,\lambda}^+\,\psi_\ep^{p_i}\, dx.
	\end{split}
\end{equation}
By \eqref{lipf} we have
\begin{equation} \label{finalf}
	\begin{split}
		\int_{\Omega_\lambda}&(f_i(\mathbf{u})-f_i(\mathbf{u}_\lambda)) 
		\,w_{i,\lambda}^+\,\psi_\ep^{p_i}\, dx=\int_{\Omega_\lambda} [f_i(u_1,u_2,...,u_m)-f_i(u_{1,\lambda},u_{2,\lambda},...,u_{m,\lambda})]\,w_{i,\lambda}^+\,\psi_\ep^{p_i}\, dx 
		\\
		&= \int_{\Omega_\lambda}
		\frac{f_i(u_1,u_2,\ldots,u_m)-f_i(u_{1,\lambda},u_2,\ldots,u_m)}
		{u_1-u_{1,\lambda}} (u_1-u_{1,\lambda}) w_{i,\lambda}^+\psi_\ep^{p_i}\, dx
		\\
		&\quad+ \int_{\Omega_\lambda}\frac{f_i(u_{1,\lambda},u_2,\ldots,u_m)-f_i(u_{1,\lambda},u_{2,\lambda},\ldots,u_m)}{u_2-u_{2,\lambda}} (u_2-u_{2,\lambda})w_{i,\lambda}^+\psi_\ep^{p_i}\, dx
		\\
		&\quad\,\,\vdots
		\\
		&\quad+ \int_{\Omega_\lambda}\frac{f_i(u_{1,\lambda},u_2,\ldots, u_i, \ldots, u_m)-f_i(u_{1,\lambda},u_{2,\lambda},\ldots, u_{i,\lambda}, \ldots, u_m)}{u_i-u_{i,\lambda}}(u_i-u_{i,\lambda})w_{i,\lambda}^+\psi_\ep^{p_i}\, dx
		\\
		&\quad\,\,\vdots
		\\
		&\quad+\int_{\Omega_\lambda}\frac{f_i(u_{1,\lambda},u_{2,\lambda}, \ldots,{u}_m)-f_i(u_{1,\lambda},u_{2,\lambda},\ldots,u_{m,\lambda})}{u_m-u_{m,\lambda}} (u_m-u_{m,\lambda}) w_{i,\lambda}^+ \psi_\ep^{p_i}\, dx\\&
		\leq
		\int_{\Omega_\lambda}
		\frac{f_i(u_1,u_2,\ldots,u_m)-f_i(u_{1,\lambda},u_2,\ldots,u_m)}
		{u_1-u_{1,\lambda}} (u_1-u_{1,\lambda})^+ w_{i,\lambda}^+\psi_\ep^{p_i}\, dx
		\\
		&\quad+ \int_{\Omega_\lambda}\frac{f_i(u_{1,\lambda},u_2,\ldots,u_m)-f_i(u_{1,\lambda},u_{2,\lambda},\ldots,u_m)}{u_2-u_{2,\lambda}} (u_2-u_{2,\lambda})^+w_{i,\lambda}^+\psi_\ep^{p_i}\, dx
		\\
		&\quad\,\,\vdots
		\\
		&\quad+ \int_{\Omega_\lambda}\frac{f_i(u_{1,\lambda},u_2,\ldots, u_i, \ldots, u_m)-f_i(u_{1,\lambda},u_{2,\lambda},\ldots, u_{i,\lambda}, \ldots, u_m)}{u_i-u_{i,\lambda}}(u_i-u_{i,\lambda})^+w_{i,\lambda}^+\psi_\ep^{p_i}\, dx
		\\
		&\quad\,\,\vdots
		\\
		&\quad+\int_{\Omega_\lambda}\frac{f_i(u_{1,\lambda},u_{2,\lambda}, \ldots,{u}_m)-f_i(u_{1,\lambda},u_{2,\lambda},\ldots,u_{m,\lambda})}{u_m-u_{m,\lambda}} (u_m-u_{m,\lambda})^+ w_{i,\lambda}^+ \psi_\ep^{p_i}\, dx
		\\
		&\leq \sum_{j=1}^{m}L_j\int_{\Omega_\lambda} w_{i,\lambda}^+ w_{j,\lambda}^+ \psi_\ep^{p_i}\, dx,
	\end{split}
\end{equation}
 where in the last inequality we thought  
 $$u_j -u_{j,\lambda}:=w_{j,\lambda}= w_{j,\lambda}^+-w_{j,\lambda}^-$$ and we used  the cooperativity 
 assumption of each $f_j$ for $j \neq i$, together with 
 the fact  that  $f_j$ is of class $C^1$ (see assumption ($h_f$)). Here, $L_j > 0$ is the Lipschitz  
 constant of $f_j$ on $[0,M]\times\cdots\times[0,M]$. Hence, resuming the computations above, we have 
  \begin{equation} \label{eq.estimfLipschitz}
   \begin{split}
   & \int_{\Omega_\lambda}(f_i(\mathbf{u})-f_i(\mathbf{u}_\lambda))
   \,w_{i,\lambda}^+\,\psi_\ep^{p_i}\, dx
   \leq \sum_{j = 1}^m L_j \int_{\Omega_\lambda} w_{j,\lambda}^+ \,w_{i,\lambda}^+ \,\psi_\ep^{p_i}\, dx.
   \end{split}
   \end{equation}
  On the other hand, by using the estimates in \eqref{eq:inequalities}, we get
  \begin{equation}\label{eq.NewBelowEstim}
  \begin{split}
  &
    \int_{\Omega_\lambda}
     \langle |\nabla u_i|^{p_i-2} \nabla u_i - |\nabla u_{i,\lambda}|^{p_i-2}\nabla u_{i,\lambda},\nabla
     w_{i,\lambda}^{+}\rangle \psi_{\ep}^{p_i}\, dx \\
     &\qquad
     \geq C_1\int_{\Omega_{\lambda}} 
     \big(|\nabla u_i| + |\nabla u_{i,\lambda}|\big)^{p_i -2}
     |\nabla w_{i,\lambda}^{+}|^2\,\psi_{\ep}^{p_i}\, dx. 
  \end{split}
  \end{equation}
  Gathering 
  together \eqref{eq.estimfLipschitz} and \eqref{eq.NewBelowEstim},
  from \eqref{eq.identityIntermediate} we then obtain
  \begin{equation} \label{eq.tostartfrom}
   \begin{split}
   & C_1\int_{\Omega_{\lambda}} 
     \big(|\nabla u_i| + |\nabla u_{i,\lambda}|\big)^{p_i -2}
     |\nabla w_{i,\lambda}^{+}|^2\,\psi_{\ep}^{p_i}\, dx \\
     & \qquad
     \leq 
     \int_{\Omega_\lambda}
     \langle |\nabla u_i|^{p_i-2} \nabla u_i - |\nabla u_{i,\lambda}|^{p_i-2}\nabla u_{i,\lambda},\nabla
     w_{i,\lambda}^{+}\rangle \psi_{\ep}^{p_i}\, dx \\
     & \qquad
     \leq p_i  \int_{\Omega_\lambda}
     \big| |\nabla u_i|^{p_i-2} \nabla u_i - |\nabla u_{i,\lambda}|^{p_i-2}\nabla u_{i,\lambda}
     \big|\,|\nabla
     \psi_{\ep}|\,w_{i,\lambda}^{+} \psi_{\ep}^{p_i -1}\, dx \\
     & \qquad\qquad 
    + \int_{\Omega_{\lambda}}\big| a_i(u_i)|\nabla u_i|^{q_i} - 
    a_i(u_{i,\lambda})|\nabla u_{i,\lambda}|^{q_i}\big| w_{i,\lambda}^{+}\psi_{\ep}^{p_i}\, dx \\
    & \qquad\qquad
    + \sum_{j = 1}^m L_j \int_{\Omega_\lambda}w_{j,\lambda}^+\,w_{i,\lambda}^+\,\psi_\ep^{p_i}\, dx \\
    & \qquad   \leq  p_i  \int_{\Omega_\lambda}
    \big| |\nabla u_i|^{p_i-2} \nabla u_i - |\nabla u_{i,\lambda}|^{p_i-2}\nabla u_{i,\lambda}
    \big|\,|\nabla
    \psi_{\ep}|\,w_{i,\lambda}^{+} \psi_{\ep}^{p_i -1}\, dx \\
     & \qquad\qquad + 
     \int_{\Omega_{\lambda}}\big|a_i(u_i)-a_i(u_{i,\lambda})\big|\,
     |\nabla u_{i,\lambda}|^{q_i}\,w_{i,\lambda}^{+}\,\psi_{\ep}^{p_i}\, dx \\
     & \qquad\qquad
     + \int_{\Omega_{\lambda}}|a_i(u_i)\big||\nabla u_i|^{q_i} - |\nabla u_{i,\lambda}|^{q_i}\big|
     w_{i,\lambda}^+\,\psi_{\ep}^{p_i}\, dx \\
     & \qquad\qquad
     + \sum_{j = 1}^m L_j \int_{\Omega_\lambda}w_{j,\lambda}^+\,w_{i,\lambda}^+\,\psi_\ep^{p_i}\, dx.
   \end{split}
  \end{equation}
  To proceed further, we now turn to provide ad-hoc estimates for the integrals in
  the right-ha\-nd side of \eqref{eq.tostartfrom}. To this end, we first introduce the notation
  \begin{align*}
  & F_i := \sum_{j = 1}^m L_j \int_{\Omega_\lambda}w_{j,\lambda}^+\,w_{i,\lambda}^+\,\psi_\ep^{p_i}\, dx.
  \end{align*}
  Moreover, we split the set $\Omega_\lambda$ as 
  $\Omega_\lambda = \Omega^{(1)}_\lambda\cup\Omega^{(2)}_\lambda,$
  where
  \begin{align*}
  & \Omega^{(1)}_\lambda = \{x\in\Omega_\lambda\setminus \Gamma_\lambda:
   \,|\nabla u_{i,\lambda}(x)|
   < 2|\nabla u_i|\}\qquad\text{and} \\[0.08cm]
  & \qquad \Omega^{(2)}_\lambda = \{
   x\in\Omega_\lambda\setminus \Gamma_\lambda:\,|\nabla u_{i,\lambda}(x)|
   \geq 2|\nabla u_i|\}.
   \end{align*}
Then it follows that
\begin{itemize}
	\item[-] By  the definition of $ \Omega^{(1)}_\lambda$, one has
	\begin{equation}\label{auxiliar0}
		|\nabla u_i| + |\nabla u_{i,\lambda} | < 3|\nabla u_i|.\end{equation}
	\item[-] By the definition  of the  set $\Omega^{(2)}_\lambda$ and standard triangular
	inequalities, one has
	\begin{equation} \label{auxiliar}
		\frac{1}{2} |\nabla u_{i,\lambda}| \leq |\nabla
		u_{i,\lambda}| - |\nabla u_i| \leq |\nabla w_{i,\lambda}| \leq |\nabla	u_{i,\lambda}| + |\nabla u_i| \leq \frac{3}{2} |\nabla
		u_{i,\lambda}|.
	\end{equation}
\end{itemize}
  Accordingly, we define
\begin{equation}
 \begin{split}
 	 & P_{i,1} :=   \int_{\Omega_\lambda^{(1)}}
 	 \big| |\nabla u_i|^{p_i-2} \nabla u_i - |\nabla u_{i,\lambda}|^{p_i-2}\nabla u_{i,\lambda}
 	 \big|\,|\nabla   \psi_{\ep}|\,w_{i,\lambda}^{+} \psi_{\ep}^{p_i -1}\, dx \\
 	 & P_{i,2} :=   \int_{\Omega_\lambda^{(2)}}
 	 \big| |\nabla u_i|^{p_i-2} \nabla u_i - |\nabla u_{i,\lambda}|^{p_i-2}\nabla u_{i,\lambda}
 	 \big|\,|\nabla   \psi_{\ep}|\,w_{i,\lambda}^{+} \psi_{\ep}^{p_i -1}\, dx \\
     & I_{i,1}:= \int_{\Omega^{(1)}_{\lambda}}\big|a_i(u_i)-a_i(u_{i,\lambda})\big|\,
     |\nabla u_{i,\lambda}|^{q_i}\,w_{i,\lambda}^{+}\,\psi_{\ep}^{p_i}\, dx \\
     & I_{i,2}:= \int_{\Omega^{(2)}_{\lambda}}\big|a_i(u_i)-a_i(u_{i,\lambda})\big|\,
     |\nabla u_{i,\lambda}|^{q_i}\,w_{i,\lambda}^{+}\,\psi_{\ep}^{p_i}\, dx \\
     & J_{i,1}:= \int_{\Omega^{(1)}_{\lambda}}
     |a_i(u_i)|\big||\nabla u_i|^{q_i} - |\nabla u_{i,\lambda}|^{q_i}\big|
     w_{i,\lambda}^+\,\psi_{\ep}^{p_i}\, dx \\
     & J_{i,2}:= \int_{\Omega^{(2)}_{\lambda}}
     |a_i(u_i)|\big||\nabla u_i|^{q_i} - |\nabla u_{i,\lambda}|^{q_i}\big|
     w_{i,\lambda}^+\,\psi_{\ep}^{p_i}\, dx.
 \end{split}
 \end{equation}
 We then turn to estimate all the above integrals. In what follows, 
 we denote by the same $C$ any positive constant
 which is independent of $\e$ (but possibly depending on $i$).
 \medskip
 
 \emph{-\,\,Estimate of $P_{i,1}$.} If $1<p_i < 2$, from \eqref{eq:inequalities}, \eqref{auxiliar0} and by H\"older's inequality we obtain
 \begin{equation}\label{eq:estP1singular}
 	\begin{split}
 		P_{i,1} & \leq C_{4} \int_{\Omega_\lambda^{(1)}} |\nabla w^+_{i,\lambda}|^{p_i-1} |\nabla
 		\psi_\varepsilon| \psi_\varepsilon^{p_i-1} w_{i,\lambda}^+ \,  dx\\
 		&\leq C_{4}
 		\left(\int_{\Omega_\lambda^{(1)}} |\nabla w_{i,\lambda}^+|^{p_i} \psi_\varepsilon^{p_i}
 		\, dx \right)^{\frac{p_i-1}{p_i}}
 		\left(\int_{\Omega_\lambda^{(1)}} |\nabla \psi_\varepsilon|^{p_i}
 		(w_{i,\lambda}^+)^{p_i} \, dx\right)^{\frac{1}{p_i}}\\
 		& \leq  C_{4} 
 		\left(\int_{\Omega_\lambda^{(1)}} (|\nabla u_i|+|\nabla
 		u_{i, \lambda}|)^{p_i} \psi_\varepsilon^{p_i} \, dx \right)^{\frac{p_i-1}{p_i}} \left( \int_{\Omega_\lambda^{(1)}} |\nabla	\psi_\varepsilon|^{p_i} (w_{i,\lambda}^+)^{p_i} \, dx \right)^{\frac{1}{p_i}}\\
 		& \leq  C_{4}  \left(3^{p_i}
 		\int_{\Omega_\lambda^{(1)}}
 		|\nabla u_i|^{p_i} \psi_\varepsilon^{p_i} \, dx  \right)^{\frac{p_i-1}{p_i}}\left( \int_{\Omega_\lambda^{(1)}} |\nabla \psi_\varepsilon|^{p_i} (w_{i,\lambda}^+)^{p_i} \, dx \right)^{\frac{1}{p_i}}\\
 		&\leq C \left(\int_{\Omega_\lambda}|\nabla u_i|^{p_i}\,dx \right)^{\frac{p_i-1}{p_i}}\left(
 		\int_{\Omega_\lambda} |\nabla \psi_\varepsilon|^{p_i} \, dx \right)^{\frac{1}{p_i}}.
 	\end{split}
 \end{equation}

\

\noindent If $p_i \geq 2$,  from \eqref{eq:inequalities} and the weighted Young's inequality we have
\[ \nonumber
\begin{split} P_{i,1}& \leq  C_{2} \int_{\Omega_\lambda^{(1)}} (|\nabla u_i| + |\nabla	u_{i,\lambda}|)^{p_i-2} |\nabla w^+_{i, \lambda}| \ |\nabla	\psi_\varepsilon| \psi_\varepsilon^{p_i-1} w_{i,\lambda}^+ \,dx\\
&\leq C\delta \int_{\Omega_\lambda^{(1)}} (|\nabla u_i| + |\nabla
u_{i,\lambda}|)^{p_i-2} |\nabla w_{i,\lambda}^+|^2 \psi_\varepsilon^{p_i}
\, dx\\
&\quad + \frac{C}{\delta} \int_{\Omega_\lambda^{(1)}} (|\nabla u_i| + \nabla	u_{i,\lambda}|)^{p_i-2} |\nabla \psi_\varepsilon|^2
\psi_\varepsilon^{p_i-2} (w_{i,\lambda}^+)^2 \, dx.
\end{split}
\]
Using \eqref{auxiliar0} and  H\"older's inequality, we obtain
\begin{equation}\label{eq:estP1degenerate}
	\begin{split}
		P_{i,1} &
		\leq C\delta \int_{\Omega_\lambda^{(1)}} (|\nabla u_i| + |\nabla u_{i, \lambda}|)^{p_i-2}
		|\nabla w_{i,\lambda}^+|^2 \psi_\varepsilon^{p_i} \, dx \\
		& \quad +\frac{C}{\delta} \int_{\Omega_\lambda^{(1)}} |\nabla u_i|^{p_i-2} 
		|\nabla \psi_\varepsilon|^2 \psi_\varepsilon^{p_i-2} (w_{i,\lambda}^+)^2 \, dx\\
		&\leq C\delta \int_{\Omega_\lambda^{(1)}} (|\nabla u_i| + |\nabla
		u_{i,\lambda}|)^{p_i-2} |\nabla w_{i,\lambda}^+|^2 \psi_\varepsilon^{p_i}
		\, dx  \\
		& \quad +  \frac{C}{\delta} 
		\left(\int_{\Omega_\lambda^{(1)}} |\nabla u_i|^{p_i} \psi_\varepsilon^{p_i}
		\, dx \right)^{\frac{p_i-2}{p_i}} \left(\int_{\Omega_\lambda^{(1)}} |\nabla
		\psi_\varepsilon|^{p_i} (w_{i,\lambda}^+)^{p_i} \, dx
		\right)^{\frac{2}{p_i}}
		\\
		&\leq C\delta \int_{\Omega_\lambda} (|\nabla u_i| + |\nabla
		u_{i,\lambda}|)^{p_i-2} |\nabla w_{i,\lambda}^+|^2 \psi_\varepsilon^{p_i}
		\, dx  \\
		&\quad +  \frac{C}{\delta} \left(\int_{\Omega_\lambda} |\nabla u_i|^{p_i}
		\, dx \right)^{\frac{p_i-2}{p_i}} \left(\int_{\Omega_\lambda} |\nabla
		\psi_\varepsilon|^{p_i} \, dx
		\right)^{\frac{2}{p_i}}.
	\end{split}
\end{equation}

\noindent  \emph{-\,\,Estimate of $P_{i,2}$.} If $1<p_i < 2$,  using the weighted Young's inequality
 and \eqref{auxiliar} we get
 
 \begin{equation}\label{eq:P2singularFinal}
 	\begin{split}
 		P_{i,2} &\leq C_{4}\int_{\Omega_\lambda^{(2)}} |\nabla w^+_{i,\lambda}|^{p_i-1} |\nabla
 		\psi_\varepsilon|  \psi_\varepsilon^{p_i-1} w_{i,\lambda}^+ \, dx\\
 		&\leq  C\delta \int_{\Omega_\lambda^{(2)}} |\nabla w_{i,\lambda}^+|^{p_i}
 		\psi_\varepsilon^{p_i} \, dx + \frac{C}{\delta} \int_{\Omega_\lambda^{(2)}}
 		|\nabla\psi_\varepsilon|^{p_i} (w_{i,\lambda}^+)^{p_i} \, dx
 		\\ 
 		& \leq  C\delta\int_{\Omega_\lambda^{(2)}} (|\nabla u_i| +
 		|\nabla u_{i, \lambda}|)^{p_i-2} \left(|\nabla u_i| + |\nabla
 		u_{i, \lambda}|\right)^2 \psi_\varepsilon^{p_i} \, dx + \frac{C}{\delta}
 		\int_{\Omega_\lambda^{(2)}} |\nabla\psi_\varepsilon|^{p_i} (w_{i,\lambda}^+)^{p_i} \,
 		dx\\
 		& \leq   C\delta\int_{\Omega_\lambda^{(2)}} (|\nabla u_i| + |\nabla
 		u_{i, \lambda}|)^{p_i-2} |\nabla u_{i, \lambda}|^2 \psi_\varepsilon^{p_i} \, dx +
 		\frac{C}{\delta} \int_{\Omega_\lambda^{(2)}} |\nabla\psi_\varepsilon|^{p_i}
 		(w_{i,\lambda}^+)^{p_i} \, dx\\
 		& \leq  C\delta  \int_{\Omega_\lambda^{(2)}}
 		(|\nabla u_i| + |\nabla u_{i, \lambda}|)^{p_i-2} |\nabla w_{i,\lambda}^+|^2
 			\psi_\varepsilon^{p_i} \, dx + \frac{C}{\delta} \int_{\Omega_\lambda^{(2)}}
 		|\nabla \psi_\varepsilon|^{p_i} (w_{i,\lambda}^+)^{p_i} \, dx\\ 
 		& \leq C\delta	 \int_{\Omega_\lambda} (|\nabla u_i| + |\nabla u_{i, \lambda}|)^{p_i-2}
 		|\nabla w^+_{i,\lambda}|^2 \psi_\varepsilon^{p_i} \, dx +
 		\frac{C}{\delta} \int_{\Omega_\lambda} |\nabla
 		\psi_\varepsilon|^{p_i} \, dx.
 	\end{split}
 \end{equation}

\noindent If $p_i \geq 2$, by the weighted Young's inequality and \eqref{eq:inequalities} we deduce that 
 \begin{equation} \nonumber
 	\begin{split}
 		P_{i,2}& \leq C_{i}\int_{\Omega_\lambda^{(2)}} (|\nabla u_i| + |\nabla u_{i, \lambda}|)^{p_i-2}
 		|\nabla w^+_{i,\lambda}| \ |\nabla \psi_\varepsilon|
 		\psi_\varepsilon^{p_i-1} w_{i,\lambda}^+ \, dx\\
 		& \leq C\delta \int_{\Omega_\lambda^{(2)}} (|\nabla u_i| + |\nabla
 		u_{i, \lambda}|)^{\frac{p_i(p_i-2)}{p_i-1}} |\nabla
 		w_{i,\lambda}^+|^{\frac{p_i}{p_i-1}} \psi_\varepsilon^{p_i}
 		\, dx+ \frac{C}{\delta} \int_{\Omega_\lambda^{(2)}} 
 		|\nabla \psi_\varepsilon|^{p_i} (w_{i,\lambda}^+)^{p_i} \, dx\\
 		&= C\delta \int_{\Omega_\lambda^{(2)}} (|\nabla u_i| + |\nabla
 		u_{i, \lambda}|)^{\frac{p_i(p_i-2)}{p_i-1}} |\nabla w_{i,\lambda}^+|^2 |\nabla
 		w_{i,\lambda}^+|^{\frac{p_i}{p_i-1}-2} \psi_\varepsilon^{p_i}
 		\, dx \\
 		& \qquad\qquad + \frac{C}{\delta} \int_{\Omega_\lambda^{(2)}} |\nabla \psi_\varepsilon|^{p_i}
 		(w_{i,\lambda}^+)^{p_i} \, dx.
 	\end{split}
 \end{equation}
 Using \eqref{auxiliar} and noticing that
 \[\frac{p_i}{(p_i-1)}-2 \leq 0,\] we obtain the following
 estimate
 \begin{equation}\label{eq:P2degenerateFinal}
 	\begin{split}
 		P_{i,2} &\leq C\delta  \int_{\Omega_\lambda^{(2)}} |\nabla
 		u_{i, \lambda}|^{p_i-2} |\nabla w_{i,\lambda}^+|^2 \psi_\varepsilon^{p_i}
 		\, dx+ \frac{C}{\delta} \int_{\Omega_\lambda^{(2)}} 
 		|\nabla \psi_\varepsilon|^{p_i} (w_{i,\lambda}^+)^{p_i} \, dx
 		\\
 		& \leq C\delta \int_{\Omega_\lambda^{(2)}} (|\nabla u_i| + |\nabla
 		u_{i, \lambda}|)^{p_i-2} |\nabla w_{i,\lambda}^+|^2 \psi_\varepsilon^{p_i}
 		\, dx + \frac{C}{\delta}
 		\int_{\Omega_\lambda^{(2)}} |\nabla \psi_\varepsilon|^{p_i} \, dx\\
 		& \leq C\delta\int_{\Omega_\lambda} (|\nabla u_i| + |\nabla
 		u_{i, \lambda}|)^{p_i-2} |\nabla w_{i,\lambda}^+|^2 \psi_\varepsilon^{p_i}
 		\, dx + \frac{C}{\delta}
 		\int_{\Omega_\lambda} |\nabla \psi_\varepsilon|^{p_i} \, dx.
 	\end{split}
 \end{equation}
In the second line of \eqref{eq:P2degenerateFinal} we exploited the fact that, since $p_i\geq 2$ then \[|\nabla
 u_{i, \lambda}|^{p_i-2}\leq (|\nabla u_i| + |\nabla
 u_{i, \lambda}|)^{p_i-2}. \]
 Collecting
 \eqref{eq:estP1singular}, \eqref{eq:estP1degenerate},
 \eqref{eq:P2singularFinal} and \eqref{eq:P2degenerateFinal}  we
 deduce that for $p_i \geq 2$ it holds 
 \begin{equation*}
  \begin{split}
   & P_{i,1} + P_{i,2} \leq C\delta 
   \int_{\Omega_{\lambda}}\big(|\nabla u_i| 
   + |\nabla u_{i,\lambda}|\big)^{p_i -2} |\nabla w_{i,\lambda}^{+}|^2 \psi_{\ep}^{p_i}\,dx \\
   &\qquad 
   + \frac{C}{\delta}\bigg(\int_{\Omega_{\lambda}}|
   \nabla u_i|^{p_i}\psi_{\ep}^{p_i}\, dx\bigg)^{\frac{p_i -2}{p_i}}
   \bigg(\int_{\Omega_\lambda}
   |\nabla \psi_{\ep}|^{p_i}|w_{i,\lambda}^{+}|^{p_i}\,dx \bigg)^{\frac{2}{p_i}} \\[0.1cm]
   & \qquad
   \qquad
   + \frac{C}{\delta}\int_{\Omega_\lambda}|\nabla \psi_{\ep}|^{p_i}\, dx
   \qquad
   (\text{for every $\delta > 0$}),
  \end{split}
 \end{equation*}
 where $C > 0$ is a suitable constant depending on $p_i,\lambda,\,\Omega$
 and $M$.

\noindent In the same way, if $1<p_i<2$, we deduce that
  \begin{equation*}
 	\begin{split}
 		& P_{i,1}+P_{i,2}\leq C\delta	 \int_{\Omega_\lambda} 
 		(|\nabla u_i| + |\nabla u_{i, \lambda}|)^{p_i-2}
 		|\nabla w^+_{i,\lambda}|^2 \psi_\varepsilon^{p_i} \, dx\\
 		&\qquad+C \left(\int_{\Omega_\lambda}|\nabla u_i|^{p_i}\,dx \right)^{\frac{p_i-1}{p_i}}\left(
 		\int_{\Omega_\lambda} |\nabla \psi_\varepsilon|^{p_i} \, dx \right)^{\frac{1}{p_i}}\\
 		& \qquad
 		\qquad
 		+ \frac{C}{\delta}\int_{\Omega_\lambda}|\nabla \psi_{\ep}|^{p_i}\, dx
 		\qquad
 		(\text{for every $\delta > 0$}),
 	\end{split}
 \end{equation*}
 where $C > 0$ is a suitable constant depending on $p_i,\lambda,\,\Omega$
 and $M$.
 
 \noindent In both cases, taking into ac\-count
 \eqref{eq:estimwplusM} and the properties of $\psi_\e$, we derive the estimate
 \begin{equation}  \label{eq:estimPEMS}
 	\begin{split}
 		& P_{i,1}+P_{i,2} \leq C\delta
 		\int_{\Omega_{\lambda}}\big(|\nabla u_i| 
 		+ |\nabla u_{i,\lambda}|\big)^{p_i -2} |\nabla w_{i,\lambda}^{+}|^2 \psi_{\ep}^{p_i}\,dx 
 		+C_\delta\e^{1/p_i},
 	\end{split}
 \end{equation}
 holding true for every choice of $\delta > 0$.
 \medskip
 
 \emph{-\,\,Estimate of $I_{i,1}$.} Since the integral $I_{i,1}$
 is actually 
 performed on the set
 $$A_{i,\lambda} = \{u_i\geq u_{i,\lambda}\}\setminus \Gamma_\lambda,$$
 from \eqref{eq:estimAilambda}, assumption $(h_a)$ and the definition
 of $\Omega^{(1)}_\lambda$ we immediately obtain
 \begin{equation} \label{eq:estimI1lambda}
  \begin{split}
  & I_{i,1} 
  \leq L\int_{\Omega^{(1)}_\lambda}
  |\nabla u_{i,\lambda}|^{q_i}\,|w_{i,\lambda}^{+}|^2\,\psi_{\ep}^{p_i}\, dx 
  \leq 2^{q_i}L\int_{\Omega^{(1)}_\lambda}|\nabla u_i|^{q_i}
  \,|w_{i,\lambda}^{+}|^2\,\psi_{\ep}^{p_i}\, dx \\
  &\qquad (\text{since $u_i\in C^{1,\alpha}(\overline{\Omega}_\lambda)$
  and $0\leq\psi_\e\leq 1$}) \\
  & \qquad 
  \leq C\int_{\Omega_\lambda}|w_{i,\lambda}^+|^2\,d x.
  \end{split}  
 \end{equation}
 
 \emph{-\,\,Estimate of $I_{i,2}$.} First of all, since also the integral 
 $I_{i,2}$ is actually performed
 on $A_{i,\lambda}$, we can use again estimate
 \eqref{eq:estimAilambda} and assumption $(h_a)$,
 obtaining
 $$I_{i,2} \leq L\int_{\Omega^{(2)}_\lambda}
 |\nabla u_{i,\lambda}|^{q_i}\,|w_{i,\lambda}^{+}|^2\,\psi_{\ep}^{p_i}\, dx \leq L\int_{\Omega^{(2)}_\lambda}
 (|\nabla u_i| +| \nabla u_{i,\lambda}|)^{q_i}\,|w_{i,\lambda}^{+}|^2\,\psi_{\ep}^{p_i}\, dx
 =: (\star).$$
 From this, since $q_i<p_i$, using \eqref{auxiliar} and by the 
 Young's inequality we obtain
 \begin{equation} \label{eq:estimI2lambda}
  \begin{split}
  & (\star)
  = L\,\int_{\Omega^{(2)}_\lambda}
 \big[(|\nabla u_i| +| \nabla u_{i,\lambda}|)^{q_i}\,\psi_\e^{q_i}\big]\cdot
 \big[|w_{i,\lambda}^{+}|^2\,\psi_{\ep}^{p_i-q_i}\big]\, dx \\
 & \qquad
 \leq C\delta\int_{\Omega^{(2)}_\lambda} (|\nabla u_i| +| \nabla u_{i,\lambda}|)^{p_i}\,\psi_\e^{p_i}\,d x
 + \frac{C}{\delta}\int_{\Omega^{(2)}_\lambda}
 |w_{i,\lambda}^+|^{\frac{2p_i}{p_i-q_i}}\,\psi_{\e}^{p_i}\,d x \\[0.1cm]
 &\qquad (\text{since
 $0 \leq \psi_\e\leq 1$ and $p_i/(p_i-q_i) > 1$}) \\
 & \qquad 
 \leq C\delta\int_{\Omega_\lambda} (|\nabla u_i| +| \nabla u_{i,\lambda}|)^{p_i-2}|\nabla w_{i,\lambda}^+|^2\,\psi_\e^{p_i}\,d x
 +\frac{C}{\delta}
 \int_{\Omega_\lambda}
 |w_{i,\lambda}^+|^{2}\,d x
 \qquad (\text{for every $\delta > 0$}).
  \end{split}
 \end{equation}  
 
 \emph{-\,\,Estimate of $J_{i,1}$.} We first observe that, since  $0\leq u_i\leq M$
 pointwise  in $\overline{\Omega}_\lambda$, by exploiting as\-sump\-tion $(h_a)$
 and the Mean Value theorem, we obtain the following estimate:
 \begin{align*}
 J_{i,1} & \leq
 C\int_{\Omega^{(1)}_\lambda}
 \big||\nabla u_i|^{q_i} - |\nabla u_{i,\lambda}|^{q_i}\big|
     w_{i,\lambda}^+\,\psi_{\ep}^{p_i}\, dx\\
 &  \leq C\,\int_{\Omega^{(1)}_\lambda}
 \big(|\nabla u_i| + |\nabla u_{i,\lambda}|\big)^{q_i-1} \,|\nabla w_{i,\lambda}^+|\, w_{i,\lambda}^+\,\psi_{\ep}^{p_i}\, dx \\
& \leq C\int_{\Omega^{(1)}_\lambda}\big(|\nabla u_i|
 + |\nabla u_{i,\lambda}|\big)^{\frac{2q_i-p_i}{2}}\big(|\nabla u_i|
 + |\nabla u_{i,\lambda}|\big)^{\frac{p_i-2}{2}}\,
     |\nabla w_{i,\lambda}^+|\,w_{i,\lambda}^+\,\psi_{\ep}^{p_i}\, dx\\
    & \leq C\delta\int_{\Omega_\lambda^{(1)}}\big(|\nabla u_i|
 + |\nabla u_{i,\lambda}|\big)^{2q_i-p_i} (|\nabla u_i| +| \nabla u_{i,\lambda}|)^{p_i-2}|\nabla w_{i,\lambda}^+|^2\,\psi_\e^{p_i}\,d x
 +\frac{C}{\delta}
 \int_{\Omega_\lambda}
 |w_{i,\lambda}^+|^{2}\,d x.
 \end{align*}
 Finally,  using \eqref{auxiliar0}, since    
  $q_i \geq\max\{p_i-1,1\}$ and therefore $q_i\geq p_i/2$,  we conclude that
 \begin{equation} \label{eq:estimJi1}
J_{i,1} \leq C\delta
 \int_{\Omega_\lambda}\big(|\nabla u_i|
 + |\nabla u_{i,\lambda}|\big)^{p_i-2}\,|\nabla w_{i,\lambda}^+|^2\,\psi_\e^{p_i}\,dx
 +\frac{C}{\delta}\int_{\Omega_\lambda}|w_{i,\lambda}^+|^2\,d x,
\end{equation}
 and this estimate holds for every choice of $\delta > 0$.
 \medskip
 
 \emph{-\,\,Estimate of $J_{i,2}$.} Using once again the fact that
 $0\leq u_i\leq M$ in $\overline{\Omega}_\lambda$, and taking into
 account assumption $(h_a)$,  we get
 \begin{align*}
 J_{i,2}
 & \leq C\int_{\Omega^{(2)}_\lambda}\big(|\nabla u_i|+|\nabla u_{i,\lambda}|\big)^{q_i}
  \,w_{i,\lambda}^+\,\psi_\e^{p_i}\,dx 
 =: (\bullet).
 \end{align*}
By Young's inequality and 
 estimate \eqref{auxiliar} we obtain
 \begin{equation} \label{eq:estimJi2}
  \begin{split}
   & (\bullet)
   = C\int_{\Omega^{(2)}_\lambda}\big[\big(|\nabla u_i|+|\nabla u_{i,\lambda}|\big)^{q_i}
  \,\psi_\e^{q_i}\big]\cdot\big[w_{i,\lambda}^+\,\psi_\e^{p_i-q_i}\big]\,dx
  \\ 
  & \qquad 
  \leq C\delta\int_{\Omega^{(2)}_\lambda}\big(|\nabla u_i|+|\nabla u_{i,\lambda}|\big)^{p_i}
  \,\psi_\e^{p_i}\,dx +\frac{C}{\delta}
  \int_{\Omega^{(2)}_\lambda}
  |w_{i,\lambda}^+|^{\frac{p_i}{p_i-q_i}}\,\psi_\e^{p_i}\,d  x \\
  & \qquad \leq C\delta\int_{\Omega_\lambda^{(2)}} \big(|\nabla u_i|+|\nabla u_{i,\lambda}|\big)^{p_i-2} |\nabla w_{i,\lambda}^+|^2
  \,\psi_\e^{p_i}\,dx +\frac{C}{\delta}\int_{\Omega_\lambda}
  |w_{i,\lambda}^+|^{2}\,d  x\\
  & \qquad  \leq C\delta
  \int_{\Omega_\lambda}\big(|\nabla u_i|
  + |\nabla u_{i,\lambda}|\big)^{p_i-2}\,|\nabla w_{i,\lambda}^+|^2\,\psi_\e^{p_i}\,dx
  +\frac{C}{\delta}\int_{\Omega_\lambda}|w_{i,\lambda}^+|^2\,d x,
  \end{split}
 \end{equation}
where we used the fact ${p_i}/{(p_i-q_i)}\geq 2$ (because $q_i\geq p_i/2$) and $0\leq u_i\leq M$ in $\overline{\Omega}_\lambda$; this estimate holds for every $\delta > 0$. \medskip
  
 \emph{-\,\,Estimate of $F_i$.}
 Since 
 $0 \leq \psi_\e\leq 1$, by Young's inequality we immediately get
 \begin{equation} \label{eq:estimFi}
 \begin{split}
  F_i\leq \sum_{j = 1}^m L_j\bigg(\int_{\Omega_\lambda}|w_{i,\lambda}^+|^2\,d x
  +
  \int_{\Omega_\lambda}|w_{j,\lambda}^+|^2\,d x\bigg)
  \leq C\sum_{j = 1}^m\int_{\Omega_\lambda}|w_{j,\lambda}^+|^2\,d x.
  \end{split}
 \end{equation}
 Thanks to all the above estimates, we can finally complete
 the proof of \eqref{eq.lemmaSumm}: in fact, by gathering together
 \eqref{eq:estimPEMS}-to-\eqref{eq:estimFi}, 
 from \eqref{eq.tostartfrom} we infer that
 \begin{equation}
  \begin{split}
  C_{1} \int_{\Omega_{\lambda}} 
     \big(|\nabla u_i| + |\nabla u_{i,\lambda}|\big)^{p_i -2}
     |\nabla w_{i,\lambda}^{+}|^2\,\psi_{\ep}^{p_i}\, dx 
     &  \leq C\delta 
   \int_{\Omega_{\lambda}}\big(|\nabla u_i| 
   + |\nabla u_{i,\lambda}|\big)^{p_i -2} |\nabla w_{i,\lambda}^{+}|^2 \psi_{\ep}^{p_i}\,dx   \\
  &  \quad +C_\delta\e^{1/p_i}
  + \frac{C}{\delta}
 \int_{\Omega_\lambda}
 |w_{i,\lambda}^+|^{2}\,d x
 \\
 & \quad + C\sum_{j = 1}^m\int_{\Omega_\lambda}|w_{j,\lambda}^+|^2\,d x,
  \end{split}
 \end{equation}
 and this estimates holds for every $\delta > 0$.
 As a consequence, if we choose $\delta$ sufficiently small
 and if we let $\varepsilon\to 0^+$ with the aid of Fatou's lemma, we obtain
 \begin{equation} \label{eq:tousePoincar}
  \int_{\Omega_{\lambda}}\big(|\nabla u_i| 
   + |\nabla u_{i,\lambda}|\big)^{p_i -2} |\nabla w_{i,\lambda}^{+}|^2\,dx    \leq C\sum_{j = 1}^m\int_{\Omega_\lambda}|w_{j,\lambda}^+|^2\,d x,
  \end{equation}
  where $C > 0$ is a suitable constant depending on $p_i, q_i, a_i, f_i, \lambda,\Omega,M$.
  This, together with the estimate in 
  \eqref{eq:estimwplusM}, immediately implies
  the desired \eqref{eq.lemmaSumm}.
  \medskip

  \noindent {\textbf{Case (ii)}.} Let us define the function
  $$A_i(t)= \int_0^t a_i(s) \, ds.$$
   Using the family of \emph{cut-off} function $\{\psi_\varepsilon\}_{\varepsilon < \e_0}$ defined in the previous case,  we now define in a very similar way, for every $\varepsilon \in (0,\e_0)$,
  the maps
  $$\varphi_{i,\varepsilon}^{(1)}(x):= 
  \begin{cases}
  	e^{-A_i(u_i(x))}w_{i,\lambda}^+(x)\,\psi_\varepsilon^{p_i}(x),
  	& \text{if $x\in\Omega_\lambda$}, \\
  	0, & \text{otherwise},
  \end{cases}$$
and 
$$\varphi_{i,\varepsilon}^{(2)}(x):= 
\begin{cases}
e^{-A_i(u_{i,\lambda}(x))}w_{i,\lambda}^+(x)\,\psi_\varepsilon^{p_i}(x),
& \text{if $x\in\Omega_\lambda$}, \\
0, & \text{otherwise}.
\end{cases}$$
It is possible to prove that, see e.g. \cite{EMM},
\begin{itemize}
	\item[(i)] $\varphi^{(1)}_{i,\varepsilon}, \varphi^{(2)}_{i,\varepsilon} 
	\in \mathrm{Lip}(\R^N);$
	
	\item[(ii)] $\operatorname{supp}(\varphi^{(1)}_{i,\varepsilon}) \subseteq \Omega_\lambda$, $\operatorname{supp}(\varphi^{(2)}_{i,\varepsilon}) \subseteq \Omega_\lambda$ and $\varphi^{(1)}_{i, \varepsilon} \equiv \varphi^{(2)}_{i,\varepsilon} \equiv 0$ near $\Gamma_\lambda$.
\end{itemize}
Hence, taking into account properties (i)-(ii) of $\varphi^{(1)}_{i,\ep}$ and $\varphi^{(2)}_{i,\ep}$,
a standard density argument allows us to use
$\varphi^{(1)}_{i,\ep}$ and $\varphi^{(2)}_{i,\ep}$ as a test functions respectively in \eqref{eq:weaksoldef} and \eqref{eq.PDEulambda}. We then subtract the latter to the former getting
  \begin{equation*}
	\begin{split}
		&
		\int_{\Omega_\lambda}
		\langle |\nabla u_i|^{p_i-2} \nabla u_i , \nabla \varphi_{i, \varepsilon}^{(1)} \rangle \, dx -\int_{\Omega_\lambda}  \langle |\nabla u_{i,\lambda}|^{p_i-2}\nabla u_{i,\lambda},\nabla
		\varphi_{i,\ep}^{(2)}\rangle\, dx \\
		&
		\qquad\qquad + \int_{\Omega_{\lambda}} a_i(u_i)|\nabla u_i|^{p_i} \varphi_{i, \varepsilon}^{(1)} \, dx- \int_{\Omega_\lambda} a_i(u_{i,\lambda})
		|\nabla u_{i,\lambda}|^{p_i} \varphi_{i,\ep}^{(2)}\, dx \\
		& \quad
		= \int_{\Omega_\lambda} f_i(\mathbf{u}) \varphi_{i,\ep}^{(1)}\, dx- \int_{\Omega_\lambda} f_i(\mathbf{u}_\lambda) \varphi_{i,\ep}^{(2)}\, dx.
	\end{split}
\end{equation*}
Now, we use the explicit expression of both $\varphi_{i, \varepsilon}^{(1)}$ and $\varphi_{i, \varepsilon}^{(2)}$ to get
\begin{equation*}
\begin{split}
	&
	-\int_{\Omega_\lambda} a_i(u_i)  e^{-A_i(u_i)} |\nabla u_i|^{p_i} w_{i, \lambda}^+ \psi_\ep^{p_i} \, dx + \int_{\Omega_\lambda} e^{-A_i(u_i)} |\nabla u_i|^{p_i-2} \langle \nabla u_i, \nabla w_{i,\lambda}^+	\rangle \psi_\varepsilon^{p_i} \,dx \\
	& \quad 
	+ p_i\int_{\Omega_\lambda} e^{-A_i(u_i)} |\nabla u_i|^{p_i-2} \langle \nabla u_i, \nabla \psi_\ep \rangle w_{i,\lambda}^+	 \psi_\varepsilon^{p_i-1} \,dx \\
	&+\int_{\Omega_\lambda}  a_i(u_{i,\lambda}) e^{-A_i(u_{i, \lambda})} |\nabla u_{i, \lambda}|^{p_i} w_{i, \lambda}^+ \psi_\ep^{p_i} \, dx
	-\int_{\Omega_\lambda} e^{-A_i(u_{i, \lambda})} |\nabla u_{i,\lambda}|^{p_i-2} \langle \nabla u_{i,\lambda},\nabla
	w_{i,\lambda}^{+}\rangle \psi_{\ep}^{p_i}\, dx \\
	& \quad - p_i \int_{\Omega_{\lambda}}e^{-A_i(u_{i, \lambda})}|\nabla u_{i,\lambda}|^{p_i-2} \langle \nabla u_{i,\lambda},\nabla
	\psi_{\ep}\rangle  w_{i,\lambda}^{+} \psi_{\ep}^{p_i -1}\, dx
	\\
	&+ \int_{\Omega_{\lambda}} a_i(u_i)|\nabla u_i|^{p_i} e^{-A_i(u_i)}w_{i,\lambda}^{+} \psi_{\ep}^{p_i} \, dx- \int_{\Omega_\lambda} a_i(u_{i,\lambda})
	|\nabla u_{i,\lambda}|^{p_i} e^{-A_i(u_{i,\lambda})}w_{i,\lambda}^{+} \psi_{\ep}^{p_i}\, dx \\
	= &\int_{\Omega_\lambda} f_i(\mathbf{u}) e^{-A_i(u_i)}w_{i,\lambda}^{+} \psi_{\ep}^{p_i}\, dx- \int_{\Omega_\lambda} f_i(\mathbf{u}_\lambda) e^{-A_i(u_{i,\lambda})}w_{i,\lambda}^{+} \psi_{\ep}^{p_i}\, dx.
\end{split}
\end{equation*}
After a simplification, we add on both sides the term
$$\int_{\Omega_{\lambda}}e^{-A_i(u_{i,\lambda})} |\nabla u_i|^{p_i -2}\langle\nabla u_i, \nabla w_{i,\lambda}^{+}\rangle \psi_{\ep}^{p_i}\, dx,$$
\noindent and, in the left hand side, we add and subtract the term 
$$p_i \int_{\Omega_{\lambda}}e^{-A_i(u_{i,\lambda})}|\nabla u_i|^{p_i -2} \langle \nabla u_i, \nabla \psi_{\ep}\rangle w_{i,\lambda}^{+} \psi_{\ep}^{p_i -1} \, dx.$$
\noindent Rearranging the terms, we find
\begin{equation}
\begin{aligned}
&\int_{\Omega_{\lambda}}e^{-A_i(u_{i,\lambda})} \langle |\nabla u_{i}|^{p_i -2} \nabla u_i - |\nabla u_{i,\lambda}|^{p_i -2} \nabla u_{i}, \nabla w_{i,\lambda}^+ \rangle \psi_{\ep}^{p_i}\, dx\\
&= \int_{\Omega_{\lambda}}\left(e^{-A_{i}(u_{i,\lambda})}- e^{-A_{i}(u_{i})}\right) |\nabla u_i|^{p_i -2} \langle \nabla u_i, \nabla w_{i,\lambda}^{+}\rangle \psi_{\ep}^{p_i}\, dx \\
& \quad + p_i \, \int_{\Omega_{\lambda}}\left(e^{-A_{i}(u_{i,\lambda})}- e^{-A_{i}(u_{i})}\right)|\nabla u_i|^{p_i -2} \langle \nabla u_i, \nabla \psi_{\ep}\rangle w_{i,\lambda}^{+}\psi_{\ep}^{p_i -1}\, dx\\
&\quad +p_i \, \int_{\Omega_{\lambda}}e^{-A_i (u_{i,\lambda})}\langle |\nabla u_{i,\lambda}|^{p_i -2}\nabla u_{i,\lambda} - |\nabla u_i|^{p_i -2}\nabla u_i ,\nabla \psi_{\ep}\rangle w_{i,\lambda}^{+}\psi_{\ep}^{p_i -1} \, dx\\
&\quad + \int_{\Omega_{\lambda}} \left(  e^{-A_i(u_i)} f_i(\mathbf{u}) - e^{-A_i(u_{i,\lambda})}f_i(\mathbf{u_{\lambda}}) \right) w_{i,\lambda}^{+}\psi_{\ep}^{p_i}  \, dx.
\end{aligned}
\end{equation}
Arguing similarly to \textbf{Case (i)}, see \eqref{eq.NewBelowEstim}, the left hand side can be estimated from below with
\begin{equation}\label{eq.Case2EstBelow}
C\int_{\Omega_{\lambda}}\left( |\nabla u_i| + |\nabla u_{i,\lambda}|\right)^{p_i -2} |\nabla w_{i,\lambda}^{+}|^2 \psi_{\ep}^{p_i}\, dx.
\end{equation}
\noindent Indeed, in the set $\Omega_{\lambda}\cap \mathrm{supp}(w_{i,\lambda}^+)$, 
one has $e^{-A_i(u_{i,\lambda})}\geq C>0$.\\
\noindent We now focus on the right hand side which we firstly bound from above passing to the absolute values, getting
\begin{equation*}
\begin{aligned}
&\int_{\Omega_{\lambda}}\left|e^{-A_{i}(u_{i,\lambda})}- e^{-A_{i}(u_{i})}\right| |\nabla u_i|^{p_i -1} |\nabla w_{i,\lambda}^{+}| \psi_{\ep}^{p_i}\, dx \\
&+ p_i \, \int_{\Omega_{\lambda}}\left|e^{-A_{i}(u_{i,\lambda})}- e^{-A_{i}(u_{i})}\right||\nabla u_i|^{p_i -1} |\nabla \psi_{\ep}| w_{i,\lambda}^{+}\psi_{\ep}^{p_i -1}\, dx\\
&+p_i \, \int_{\Omega_{\lambda}}e^{-A_i (u_{i,\lambda})}\left| |\nabla u_{i,\lambda}|^{p_i -2}\nabla u_{i,\lambda} - |\nabla u_i|^{p_i -2}\nabla u_i \right| |\nabla \psi_{\ep}| w_{i,\lambda}^{+}\psi_{\ep}^{p_i -1} \, dx\\
&+\int_{\Omega_{\lambda}} \left(  e^{-A_i(u_i)} f_i(\mathbf{u}) - e^{-A_i(u_{i,\lambda})}f_i(\mathbf{u_{\lambda}}) \right) w_{i,\lambda}^{+}\psi_{\ep}^{p_i}  \, dx
 =: \mathcal{J}_{i,1} + \mathcal{J}_{i,2} + \mathcal{J}_{i,3} + \mathcal{J}_{i,4}.
\end{aligned}
\end{equation*}
For the reader's convenience, we recall that, in this case, we are assuming $2\leq p_i\leq N$. We have the following

 \emph{-\,\,Estimate of $\mathcal{J}_{i,1}$.} Using the lipschitzianity of $t \mapsto e^{-A_i(t)}$, we get that there exists a positive constant $C>0$ such that
 \begin{equation*}
 \begin{aligned}
 \mathcal{J}_{i,1} &\leq C \int_{\Omega_{\lambda}}|\nabla u_i|^{p_i -1}|\nabla w_{i,\lambda}^+| \psi_\ep^{\frac{p_i}{2}} 
 w_{i,\lambda}^+  \psi_\ep^{\frac{p_i}{2}} \, dx.  
 \end{aligned}
 \end{equation*}
 \noindent Now, by exploiting the weighted Young's inequality, we obtain that for every $\delta>0$ it holds the following
 \begin{equation}\label{eq.StimaJi1}
 	\begin{split}
 	\mathcal{J}_{i,1} &\leq C\delta \int_{\Omega_{\lambda}} |\nabla u_i|^{2p_i -2}|\nabla w_{i,\lambda}^{+}|^{2} \psi_{\ep}^{p_i}\, dx + \dfrac{C}{\delta} \int_{\Omega_{\lambda}} (w_{i,\lambda}^+)^{2} \psi_{\ep}^{p_i}\, dx \\
 	&\leq C\delta \int_{\Omega_{\lambda}} (|\nabla u_i|+|\nabla u_{i,\lambda}|)^{p_i -2}|\nabla w_{i,\lambda}^{+}|^{2} \psi_{\ep}^{p_i}\, dx + \dfrac{C}{\delta} \int_{\Omega_{\lambda}} (w_{i,\lambda}^+)^{2} \psi_{\ep}^{p_i}\, dx
 \end{split} 	
 \end{equation}
 where in the last step $C>0$ depends also on $\|\nabla u\|_{L^{\infty}(\Omega_{\lambda})}$.

 \emph{-\,\,Estimate of $\mathcal{J}_{i,2}$.} Using once again the 
 lipschitzianity of $t \mapsto e^{-A_i(t)}$ together with the boundedness of $\nabla u$ in $\Omega_{\lambda}$, and then exploiting the H\"older's inequality with  exponents $(p_i, p_i/(p_i-1))$, we get
 \begin{equation}
 \mathcal{J}_{i,2} \leq C\left( \int_{\Omega_{\lambda}}|\nabla \psi_{\ep}|^{p_i} (w_{i,\lambda}^{+})^{p_i} \, 
 dx \right)^\frac{1}{p_i}\left( \int_{\Omega_{\lambda}}\psi_{\ep}^{p_i} \, dx\right)^\frac{p_i}{p_i-1}.
 \end{equation}
 
  \emph{-\,\,Estimate of $\mathcal{J}_{i,3}$.} We notice first that 
  $e^{-A_i (u_{i,\lambda})} \leq 1$. Arguing similarly  as we did  in the com\-pu\-ta\-tions that led to \eqref{eq:estP1singular} and subsequently to 
  \eqref{eq:estimPEMS},  we deduce:  \begin{equation*}
  \mathcal{J}_{i,3} \leq  C\delta
  \int_{\Omega_{\lambda}}\big(|\nabla u_i| 
  + |\nabla u_{i,\lambda}|\big)^{p_i -2} |\nabla w_{i,\lambda}^{+}|^2 \psi_{\ep}^{p_i}\,dx 
  +C_\delta\e^{1/p_i} \quad(\text{for every $\delta>0$}).
  \end{equation*}
  
   \emph{-\,\,Estimate of $\mathcal{J}_{i,4}$.} We first consider the function 
   $$g_i (t) := e^{-A_i (t_i)} f_i (t), \quad t\in \mathcal{I}^m,$$
   \noindent which is still a $C^1$ function satisfying the cooperativity condition
   $\de_{t_k}g_i \geq 0$ on $\mathcal{I}^m$ for every $k\neq i$.  
   Hence, we can repeat the computations made in \eqref{lipf}  up to \eqref{eq.estimfLipschitz} to get
 \begin{equation}
 \mathcal{J}_{i,4} \leq \sum_{j = 1}^m L_{g_j} \int_{\Omega_\lambda} w_{j,\lambda}^+ \,w_{i,\lambda}^+ \,\psi_\ep^{p_i}\, dx \leq  C\sum_{j = 1}^m\int_{\Omega_\lambda}|w_{j,\lambda}^+|^2\,d x.
 \end{equation}
Putting everything together, we can conclude as in the previous \textbf{Case (i)}. Hence, we obtain the desired \eqref{eq.lemmaSumm} for every $i=1,\ldots, m$.
\end{proof}
We conclude this section by proving the following useful lemma.

\begin{lem}\label{lem:SobolevSplitting}
Let assumptions $(h_\Omega)$-to-$(h_f)$ be in force. 	Let $i \in\{1,\ldots,m\}$ be fixed, $1<p_i <2$, and let $\lambda\in(\varrho,0)$.
	Then, there exists a constant $\mathbf{c} = \mathbf{c}_i > 0$ such that
	\begin{equation}\label{eq:lemThesis}
		\int_{\Omega_\lambda} |\nabla w_{i,\lambda}^+|^{p_i} \, dx \leq  \mathbf{c} \ \left(\int_{\Omega_\lambda} (|\nabla u_i| + |\nabla u_{i,\lambda}|)^{p_i-2}| \nabla w_{i,\lambda}^+|^2 \, dx\right)^{\frac{p_i}{2}}.
	\end{equation}
\end{lem}

\begin{proof}	Let us define the set $\Omega_\lambda^+:=\Omega_\lambda \cap \,\text{supp}(w^+_{i,\lambda})$ and let $\psi_{\ep}$ be the cut-off function defined in \eqref{test1}. 
Using  H\"older's inequality with conjugate exponents $({2}/{(2-p_i)},{2}/{p_i})$,
we obtain
	\begin{equation}\label{eq:SobolevSplitting}
		\begin{split}
		\int_{\Omega_\lambda^+} &|\nabla w_{i,\lambda}^+|^{p_i}\psi_{\ep}^{p_i}\ \, dx \\& = \int_{\Omega_\lambda^+} (|\nabla u_i|+|\nabla u_{i,\lambda}|)^{\frac{p_i(2-p_i)}{2}} (|\nabla u_i|+|\nabla u_{i,\lambda}|)^{\frac{p_i(p_i-2)}{2}}|\nabla w_{i,\lambda}^+|^{p_i} \psi_{\ep}^{p_i}\, dx\\
		& \leq \left(\int_{\Omega_\lambda^+} (|\nabla u_i|+|\nabla u_{i,\lambda}|)^{p_i} \psi_{\ep}^{p_i}\,dx\right)^{\frac{2-p_i}{p_i}} \left(\int_{\Omega_\lambda^+} (|\nabla u_i|+|\nabla u_{i,\lambda}|)^{p_i-2} |\nabla w_{i,\lambda}^+|^2 \psi_{\ep}^{p_i} \, dx\right)^{\frac{p_i}{2}}.
		\end{split}
	\end{equation}
Now we are going to give an estimate to the first term in the right hand side of \eqref{eq:SobolevSplitting}, i.e. the term
$$\mathcal{I}:=\int_{\Omega_\lambda^+} (|\nabla u_i|+|\nabla u_{i,\lambda}|)^{p_i}\psi_{\ep}^{p_i} \,dx.$$
In the same spirit of the previous lemma, we split the set $\Omega_\lambda$ as 
$\Omega_\lambda = \Omega^{(1)}_\lambda\cup\Omega^{(2)}_\lambda,$
where
\begin{align*}
	& \Omega^{(1)}_\lambda = \{x\in\Omega_\lambda^+\setminus \Gamma_\lambda:
	\,|\nabla u_{i,\lambda}(x)|
	< 2|\nabla u_i|\}\qquad\text{and} \\[0.08cm]
	& \qquad \Omega^{(2)}_\lambda = \{
	x\in\Omega_\lambda^+\setminus \Gamma_\lambda:\,|\nabla u_{i,\lambda}(x)|
	\geq 2|\nabla u_i|\}.
\end{align*}
Hence
$$\mathcal{I}=\int_{\Omega_\lambda^{(1)}} (|\nabla u_i|+|\nabla u_{i,\lambda}|)^{p_i} \psi_{\ep}^{p_i}\,dx + \int_{\Omega_\lambda^{(2)}} (|\nabla u_i|+|\nabla u_{i,\lambda}|)^{p_i} 
\psi_{\ep}^{p_i}\,dx =: \mathcal{I}_1+\mathcal{I}_2.$$
\

\textit{- Estimate of $\mathcal{I}_1$.} Since $\psi_{\ep}^{p_i}\leq 1$, using 
\eqref{auxiliar0} we immediately get
\begin{equation}\label{eq:I1lemmata}
	\mathcal{I}_1 \leq C\int_{\Omega_\lambda^{(1)}} |\nabla u_i|^{p_i} \,dx,
\end{equation}

\textit{- Estimate of $\mathcal{I}_2$.} Using \eqref{auxiliar}, we obtain
\begin{equation}\label{eq:I2lemmata}
\begin{split}
\mathcal{I}_2 &= \int_{\Omega_\lambda^{(2)}} (|\nabla u_i|+|\nabla u_{i,\lambda}|)^{p_i-2} (|\nabla u_i|+|\nabla u_{i,\lambda}|)^2 \psi_{\ep}^{p_i}\, dx\\
&\leq C
\int_{\Omega_\lambda^{(2)}} (|\nabla u_i|+|\nabla u_{i,\lambda}|)^{p_i-2} |\nabla u_{i,\lambda}|^2\psi_{\ep}^{p_i} \, dx\\
&\leq C \int_{\Omega_\lambda^{(2)}} (|\nabla u_i|+|\nabla u_{i,\lambda}|)^{p_i-2} |\nabla w_{i,\lambda}^+|^2\psi_{\ep}^{p_i} \, dx
\end{split}
\end{equation}
Collecting \eqref{eq:I1lemmata} and \eqref{eq:I2lemmata} in order to estimate $\mathcal{I}$, we have
$$\mathcal{I} \leq C\bigg(\int_{\Omega_\lambda^{(1)}} |\nabla u_i|^{p_i} \,dx + 
\int_{\Omega_\lambda^{(2)}} (|\nabla u_i|+|\nabla u_{i,\lambda}|)^{p_i-2} 
|\nabla w_{i,\lambda}^+|^2 \psi_{\ep}^{p_i}\, dx\bigg).$$
Now, since each $\nabla u_i \in L^\infty(\Omega_\lambda)$ for every $\lambda <0$ and the second term is bounded by Lemma~\ref{leaiuto}, we deduce that 
\[\mathcal{I} \leq C,\] 
where $C$ is a positive constant that does not depends on $\varepsilon$. Since $\psi_{\ep}^{p_i}\leq 1$, the inequality \eqref{eq:SobolevSplitting} becomes
\begin{equation}\label{eq:SobolevSplitting1}		
\int_{\Omega_\lambda^+} |\nabla w_{i,\lambda}^+|^{p_i}\psi_{\ep}^{p_i}\ \, dx \leq C
 \left(\int_{\Omega_\lambda^+} (|\nabla u_i|+|\nabla u_{i,\lambda}|)^{p_i-2} |\nabla w_{i,\lambda}^+|^2 \, dx\right)^{\frac{p_i}{2}},
	\end{equation}
	Finally by Fatou's lemma from \eqref{eq:SobolevSplitting1}, we get  \eqref{eq:lemThesis}. 
\end{proof}

 \section{Proof of Theorem \ref{thm:symmetry}} \label{sec:proofThm}

 Thanks to all the results established so far, we are ready to give the
 \begin{proof}[Proof of Theorem \ref{thm:symmetry}] 
 Our approach relies on a suitable adaptation of the integral version
 of the moving plane method.
  We consider the set
 $$\Lambda := \big\{\eta \in (\varrho,0):\,\text{$u_i\leq u_{i,\lambda}$
 on $\Omega_\lambda\setminus R_\lambda(\Gamma)$ for all $\lambda\in(\varrho,\eta]$ and for all
 $i = 1,\ldots,m$}\big\},$$
 and we claim that the following facts hold:
 \begin{itemize}
  \item[(a)] $\Lambda\neq \varnothing$;
  \vspace*{0.05cm}
  
  \item[(b)] setting $\lambda_0 := \sup(\Lambda)$, one has $\lambda_0 = 0$.
 \end{itemize}
 \medskip
 
 \noindent \textbf{Proof of (a).} First of all we observe that, since $\Gamma$ is compact
 and contained in $\Omega\cap\Pi_0$,
 it is possible to find a small $\tau_0 > 0$ such that
 $\varrho+\tau_0 < 0$ and
 $$R_\lambda(\Gamma)\cap\overline{\Omega}_\lambda = \varnothing
 \qquad\text{for every $\lambda\in I_0 := (\varrho,\varrho+\tau_0]$}.$$
 In particular, for every $\lambda\in I_0$ we have 
 $\mathbf{u},\,\mathbf{u}_\lambda\in C^{1,\alpha}(\overline{\Omega}_\lambda)$.
 On the other hand, since both $u_i$ and $u_{i,\lambda}$ are non-negative
 (for every $i = 1,\ldots,m$), it is immediate to recognize that
 $$\text{$u_i\leq u_{i,\lambda}$ on $\de\Omega_\lambda$}\qquad
 \text{for every $i = 1,\ldots,m$}.$$
 Now, it is possible to start the moving plane procedure  using \cite[Proposition 2.5]{EMM}. For the 
 reader's convenience we  state such a result adapted to 
 our context. The proof is exactly the same and 
 therefore we skip it. 
 \begin{prop} \label{prop:WCPsmall}
 Let assumptions $(h_a)$ and $(h_f)$ be in force, and suppose that
 $$p_i>1\quad\text{and}\quad q_i \geq  \max\{p_i -1 , 1\}\qquad\text{
 for every $i = 1,\ldots,m$}.$$
 In addition, suppose that
 $\mathbf{u},\,
 \mathbf{u}_\lambda\in C^{1,\alpha}(\overline \Omega_{\varrho+\tau_0})$. Then,
 there exists a number $\delta > 0$, de\-pen\-ding on 
 $m, p_i, q_i, a_i, f_i, \|u_i\|_{L^{\infty}(\Omega_{\varrho+\tau_0})}, \|\nabla u_i\|_{L^{\infty}(\Omega_{\varrho+\tau_0})}$ and $\|\nabla \tilde{u}_i\|_{L^{\infty}(\Omega_{\varrho+\tau_0})}$,
 with the following property: if $\Omega_\lambda \subseteq \Omega_{\varrho+\tau_0}$ 
 is such that
 $|\Omega_\lambda|\leq \delta$, and if 
 $$\text{$u_i \leq {u}_{i,\lambda}$ on $\partial \Omega_\lambda$}\quad
 \text{for every $i=1,\ldots,m$},$$
 then
 $u_i \leq {u}_{i,\lambda}$ in $\Omega_\lambda$ for every $i=1,\ldots,m$.
\end{prop}
Therefore, by possibly shrinking $\tau_0$ in such a way that
 $|\Omega_\lambda|\leq \delta$ for every $\lambda\in I_0$, we derive that
 $u_i\leq u_{i,\lambda}$ in $\Omega_\lambda$
 for every $i = 1,\ldots,m$.
 Hence, $\eta := \varrho+\tau_0\in\Lambda$, and thus
 $$\Lambda\neq \varnothing.$$
 
 \noindent\textbf{Proof of (b).} On account of (a), $\lambda_0$ is well-defined
 and $\lambda_0\leq 0$. Arguing by contradiction, we then suppose that
 $\lambda_0 < 0$, and we prove that there exists some $\tau_0 > 0$ such that
 \begin{equation} \label{eq:contradictionlambdazero}
  \text{$u_{i}\leq u_{i,\lambda}$ on $\Omega_\lambda\setminus R_\lambda(\Gamma)$}
  \quad\text{for all $i = 1,\ldots,m$ and $\lambda\in (\lambda_0,\lambda_0+\tau_0]$}.
 \end{equation}
 Since 
 \eqref{eq:contradictionlambdazero} is clearly in contrast with the
 very definition of $\lambda_0$, we can conclude that
 $$\lambda_0 = 0.$$
 In order to establish the needed \eqref{eq:contradictionlambdazero}, we proceed as follows:
 first of all,
 since both $\mathbf{u}$ and $\mathbf{u}_\lambda$
 are continuous on $\Omega_\lambda\setminus R_\lambda(\Gamma)$, 
  we observe that 
 \begin{equation} \label{eq:segnodeboleWMP}
  \text{$u_i\leq u_{i,\lambda_0}$ on $\Omega_{\lambda_0}\setminus R_{\lambda_0}(\Gamma)$}\qquad
 \text{for every $i = 1,\ldots,m$}.
 \end{equation}
 We then claim that, as a consequence of
 \eqref{eq:segnodeboleWMP}, one actually has
 \begin{equation} \label{eq:toprovewithSMPCaseI} 
  \text{$u_i < u_{i,\lambda_0}$ on $\Omega_{\lambda_0}\setminus R_{\lambda_0}(\Gamma)$}
  \qquad
 \text{for every $i = 1,\ldots,m$}.
 \end{equation}
 To prove \eqref{eq:toprovewithSMPCaseI}, we
 arbitrarily fix $i\in\{1,\ldots,m\}$. 

 We first point out that, since
 $\Omega_{\lambda_0} = \Omega\cap \{x_1<\lambda_0\}$ 
 is connected and the com\-pact set 
 $\Gamma_{\lambda_0} := R_{\lambda_0}(\Gamma)$ has vanishing $p_i$-capacity
 (as the same is true of $\Gamma$, see assumption $(h_\Gamma)$), it is 
 not difficult to check that
 $\Omega_{\lambda_0}\setminus \Gamma_{\lambda_0}$
 is \emph{connected} (see, e.g., \cite[Lemma 2.4]{BVV}); 
 thus, owing to assumptions $(h_f)$, we deduce that it holds the following distributional inequality 
$$-\Delta_{p_i} u_i + a_i(u_i) |\nabla u_i|^{q_i} + \Lambda_i u_i \leq -\Delta_{p_i} u_{i,\lambda_0} + a_i(u_{i,\lambda_0})
  |\nabla u_{i,\lambda_0}|^{q_i} + \Lambda_i u_{i,\lambda_0} \quad\text{in } \Omega_{\lambda_0}
  \setminus\Gamma_{\lambda_0},$$
where $\Lambda_i$ is a positive constant.
 As a con\-se\-quen\-ce, since  $u_i\leq u_{i,\lambda_0}$
 in $\Omega_{\lambda_0}\setminus \Gamma_{\lambda_0}$, we can apply the Strong Comparison 
 Principle (see, e.g., \cite[Theorem 1.2]{Mon}), ensuring 
 that
 \begin{equation} \label{eq:dicSCP}
  \text{either $u_i\equiv u_{i,\lambda_0}$ or $u_i < u_{i,\lambda_0}$ \quad in $\Omega_{\lambda_0}\setminus
 	\Gamma_{\lambda_0}$}.
 \end{equation}
 On the other hand, since $\mathbf{u}$ solves
 \eqref{eq:System}, we have $u_i-u_{i,\lambda_0} = -u_{i,\lambda_0} < 0$ on 
 $\de\Omega_{\lambda_0}\cap\de\Omega$. This
 immediately gives \eqref{eq:toprovewithSMPCaseI}, since the alternative 
 $u_i\equiv u_{i,\lambda_0}$ in $\Omega_{\lambda_0}\setminus \Gamma_{\lambda_0}$
 obviously cannot be achieved  because of the Dirichlet boundary condition and the fact that $u_i>0$ in $\Omega \setminus \Gamma$.
 \medskip

 \noindent Now we have fully established 
 \eqref{eq:toprovewithSMPCaseI}, we can continue with the proof
 of \eqref{eq:contradictionlambdazero}. 
 To begin with, we arbitrarily fix a compact set
 $\mathcal{K}\subseteq 
 \Omega_{\lambda_0}\setminus R_{\lambda_0}(\Gamma)$ and we observe that,
 since $R_{\lambda_0}(\Gamma)$ 
 is compact, we can find $\tau_0 = \tau_0(\mathcal{K}) > 0$ so small that
 $$\mathcal{K}\subseteq \Omega_\lambda\setminus R_\lambda(\Gamma)\qquad
 \text{for every $\lambda_0\leq \lambda\leq \lambda_0+\tau_0$}.$$
 Moreover, since $\mathbf{u}$ is continuous on $\mathcal{K}$, a simple
 uniform-continuity argument based on \eqref{eq:toprovewithSMPCaseI} shows that,
 by possibly shrinking $\tau_0$, we also have (for all $i = 1,\ldots,m$)
 \begin{equation} \label{eq:uilequilambdaK}
  \text{$u_i < u_{i,\lambda}$ on $\mathcal{K}$}
 \qquad\text{for every $\lambda_0\leq \lambda\leq \lambda_0+\tau_0$}.
 \end{equation}
 We now turn to prove that, for every $\lambda_0\leq\lambda\leq\lambda_0+\tau_0$
 and every $i = 1,\ldots,m$, one can find a constant
 $C_i > 0$, depending of $p_i, q_i, a_i, f_i, \lambda,\Omega,M$ (with  $M = M_{\mathbf{u}} := \max_{1\leq j\leq m}
 \big(\|u_j\|_{L^\infty(\overline{\Omega}_{\rho+\tau_0})}+
 \|\nabla u_j\|_{L^\infty(\overline{\Omega}_{\rho+\tau_0})}\big)
 < +\infty$),
 such that
 \begin{equation} \label{eq:integralestimatesperPoincare}
  \int_{\Omega_{\lambda}\setminus\mathcal{K}}\big(|\nabla u_i| 
   + |\nabla u_{i,\lambda}|\big)^{p_i -2} |\nabla w_{i,\lambda}^{+}|^2\,dx 
   \leq C\sum_{j = 1}^m\int_{\Omega_{\lambda}\setminus\mathcal{K}}|w_{j,\lambda}^+|^2\,d x.
 \end{equation}
 Taking \eqref{eq:integralestimatesperPoincare} for granted for a moment,
 let us show how this integral estimate can be used
 in order to prove \eqref{eq:contradictionlambdazero}.
 First of all, by taking the sum in \eqref{eq:integralestimatesperPoincare} 
 for $i=1,\ldots,m$, we get
 \begin{equation}\label{eq:nnsannochisonoio}
 \sum_{i = 1}^m\int_{\Omega_{\lambda}\setminus\mathcal{K}}\big(|\nabla u_i| 
   + |\nabla u_{i,\lambda}|\big)^{p_i -2} |\nabla w_{i,\lambda}^{+}|^2\,dx 
   \leq C'\sum_{i = 1}^m\int_{\Omega_{\lambda}\setminus\mathcal{K}}|w_{i,\lambda}^+|^2\,d x.
 \end{equation}  
Now, we have to distinguish the singular case from the degenerate one. To this end, let us
suppose (up to a rearrangement of the sum in both sides of  the inequality  \eqref{eq:nnsannochisonoio})
that
$$\frac{(2N+2)}{(N+2)}<p_i \leq 2\quad\text{for every $i = 1,\ldots,m'$},$$
for some $1\leq m'\leq m$.
In this case,  we have that  $p_i^* > 2$. Applying  
the H\"older inequality with conjugate exponents $((p^*_i-2)/p^*_i,p_i^*/2)$ and 
the Sobolev inequality to the first $m'$ terms of the right hand side of 
\eqref{eq:nnsannochisonoio}, we  deduce that
\begin{equation}\label{eq:finalPoincare}
\begin{split}
&\sum_{i = 1}^{m'}\int_{\Omega_{\lambda}\setminus\mathcal{K}} \big(|\nabla u_i|  + |\nabla u_{i,\lambda}|\big)^{p_i -2} |\nabla w_{i,\lambda}^{+}|^2\,dx + \sum_{i = m'+1}^m\int_{\Omega_{\lambda}\setminus\mathcal{K}}\big(|\nabla u_i|  + |\nabla u_{i,\lambda}|\big)^{p_i -2} |\nabla w_{i,\lambda}^{+}|^2\,dx  \\  
&\leq C\sum_{i = 1}^{m'}\int_{\Omega_{\lambda}\setminus\mathcal{K}}|w_{i,\lambda}^+|^2\,d x + 
 C\sum_{i = m'+1}^m\int_{\Omega_{\lambda}\setminus\mathcal{K}}|w_{i,\lambda}^+|^2\,d x\\
& \leq  C\sum_{i = 1}^{m'} |\Omega_{\lambda}\setminus\mathcal{K}|^{\frac{p^*_i-2}{p^*_i}}
\left(\int_{\Omega_{\lambda}\setminus\mathcal{K}}|w_{i,\lambda}^+|^{p^*_i}\,d x\right)^{\frac{2}{p^*_i}} + C\sum_{i = m'+1}^m\int_{\Omega_{\lambda}\setminus\mathcal{K}}|w_{i,\lambda}^+|^2\,d x\\
& \leq  C\sum_{i = 1}^{m'} |\Omega_{\lambda}\setminus\mathcal{K}|^{\frac{p^*_i-2}{p^*_i}}
\left(\int_{\Omega_{\lambda}\setminus\mathcal{K}}|\nabla w_{i,\lambda}^+|^{p_i}\,d x\right)^{\frac{2}{p_i}} + C\sum_{i = m'+1}^m\int_{\Omega_{\lambda}\setminus\mathcal{K}}|w_{i,\lambda}^+|^2\,d x.
\end{split}
\end{equation}
In the first $m'$ terms of the right hand side of \eqref{eq:finalPoincare} we apply Lemma \ref{lem:SobolevSplitting},  
while in the last terms (from $m'+1$ to $m$) of the same inequality, since 
$p_i \geq 2$,  we  do apply the weighted Sobolev inequality  \cite[Theorem 2.3]{Mon}. Hence we obtain
\begin{equation}\label{eq:finalPoincareBis}
\begin{split}
&	\sum_{i = 1}^{m'}\int_{\Omega_{\lambda}\setminus\mathcal{K}}\big(|\nabla u_i|  + |\nabla u_{i,\lambda}|
\big)^{p_i -2} |\nabla w_{i,\lambda}^{+}|^2\,dx + 
\sum_{i = m'+1}^m\int_{\Omega_{\lambda}\setminus\mathcal{K}}\big(|\nabla u_i|  + |\nabla u_{i,\lambda}
\big)^{p_i -2} |\nabla w_{i,\lambda}^{+}|^2\,dx  \\
& \qquad\leq C\sum_{i = 1}^{m'}|\Omega_{\lambda}\setminus\mathcal{K}|^{\frac{p^*_i-2}{p^*_i}}\int_{\Omega_{\lambda}\setminus\mathcal{K}}\big(|\nabla u_i|  + |\nabla u_{i,\lambda}|\big)^{p_i -2} |\nabla w_{i,\lambda}^{+}|^2\,dx \\
& \qquad\quad\quad + C\sum_{i = m'+1}^m C_{P}(\Omega_{\lambda}\setminus\mathcal{K}) \int_{\Omega_{\lambda}\setminus\mathcal{K}}|\nabla u_i|^{p_i -2} |\nabla w_{i,\lambda}^{+}|^2\,dx\\
& \qquad\leq C\sum_{i = 1}^{m'} |\Omega_{\lambda}\setminus\mathcal{K}|^{\frac{p^*_i-2}{p^*_i}}\int_{\Omega_{\lambda}\setminus\mathcal{K}}\big(|\nabla u_i|  + |\nabla u_{i,\lambda}|\big)^{p_i -2} |\nabla w_{i,\lambda}^{+}|^2\,dx \\
& \qquad\quad\quad + C\sum_{i = m'+1}^m C_{P}(\Omega_{\lambda}\setminus\mathcal{K}) \int_{\Omega_{\lambda}\setminus\mathcal{K}}\big(|\nabla u_i|  + |\nabla u_{i,\lambda}|\big)^{p_i -2} |\nabla w_{i,\lambda}^{+}|^2\,dx,
\end{split}
\end{equation}
where $C_{P}(\Omega_{\lambda}\setminus\mathcal{K})$ is  the Poincar\'e constant that tends to zero, when the Lebesgue measure $|\Omega_{\lambda}\setminus\mathcal{K}|$ tends to zero. From \eqref{eq:finalPoincareBis}, up to redefine constants we have 
\begin{equation}\label{eq:finalPoincareBisBis}
\begin{split}
&	\sum_{i = 1}^{m}\int_{\Omega_{\lambda}\setminus\mathcal{K}}\big(|\nabla u_i|  + |\nabla u_{i,\lambda}|\big)^{p_i -2} |\nabla w_{i,\lambda}^{+}|^2\,dx   \\
&\qquad\leq C(\Omega_{\lambda}\setminus\mathcal{K})\sum_{i = 1}^m \int_{\Omega_{\lambda}\setminus\mathcal{K}}\big(|\nabla u_i|  + |\nabla u_{i,\lambda}|\big)^{p_i -2} |\nabla w_{i,\lambda}^{+}|^2\,dx,
\end{split}
\end{equation}
where $C(\Omega_{\lambda}\setminus\mathcal{K})$
tends to zero, when the Lebesgue measure $|\Omega_{\lambda}\setminus\mathcal{K}|$ tends to zero. Now,  we 
choose the compact set $\mathcal{K}$ sufficiently large such that
$$C(\Omega_{\lambda}\setminus\mathcal{K})< 1.$$
From this fact we immediately deduce that $u_i \leq u_{i,\lambda}$ for each $\lambda_0 < \lambda \leq 
\lambda_0+\tau$ and this gives a contradiction with the  definition of $\lambda_0$. 
Therefore $\lambda_0 
=0$ and the thesis (i) is proved. Since the moving plane procedure can be performed in the same way
but in the opposite direction, then this proves the desired symmetry
result. To prove the  thesis (ii), we observe that  the monotonicity of the solution  is in fact  implicit in the moving plane method and in particular we get that $\partial_{x_1} u_i \geq 0$  in  $\Omega_0$.  To get \eqref{eq:ehvabe} it is sufficient to apply the Strong Maximum Principle  
\cite[Theorem 2.3]{EMM}.
\

Hence, we are left to prove \eqref{eq:integralestimatesperPoincare}.
To this end, for every \emph{fixed} $\lambda\in(\lambda_0,\lambda_0+\tau_0]$ we choose an
open neighborhood $\mathcal{O}_\lambda\subseteq\Omega_\lambda\setminus\mathcal{K}$
of $\Gamma_\lambda = R_\lambda(\Gamma)$, and a \emph{cut-off family} $\{\psi_\e\}_{\e<\e_0}$  for $\Gamma_\lambda$ related with   $\mathcal{O}_\lambda$. This means, precisely, that
\begin{enumerate}
  \item $\psi_\varepsilon\in \mathrm{Lip}(\R^N)$ and $0\leq \psi_\e\leq 1$ pointwise in $\R^N$;
  \item there exists
  an open neighborhood $\mathcal{V}^{\lambda}_\e\subseteq\mathcal{O}_\lambda$ of $\Gamma_\lambda$ 
  such that
  $$\text{$\psi_\e\equiv 0$ on $\mathcal{V}^{\lambda}_\e$};
  $$
  \item $\psi_\e(x)\to 1$ as $\e\to 0^+$ for a.e.\,$x\in\R^N$;
  \item there exists a constant $C_0 > 0$, independent of $\e$, such that
  $$\int_{\R^N}|\nabla \psi_\e|^{p_i}\,d x \leq C_0\varepsilon.$$
 \end{enumerate}
 We also fix $i \in \{1,\ldots,m\}$,
 and we distinguish two cases.
\begin{itemize}
   \item[(i)] $q_i < p_i$;
   \vspace*{0.1cm}
   \item[(ii)]  $q_i = p_i$.
  \end{itemize}
  \vspace*{0.1cm}
  
  \noindent \textbf{Case (i)}. In this case, for every $\varepsilon \in (0,\e_0)$ we consider
  the map
  $$\varphi_{i,\varepsilon}(x):= 
  \begin{cases}
  	w_{i,\lambda}^+(x)\,\psi_\varepsilon^{p_i}(x) = (u_i-u_{i,\lambda})^+(x)\,\psi_\varepsilon^{p_i}(x),
  	& \text{if $x\in\Omega_\lambda$}, \\
  	0, & \text{otherwise}.
  \end{cases}$$
 As already recognized in Lemma \ref{leaiuto}, $\varphi_{i,\e}$ satisfies the following properties:
 \begin{itemize}
  \item[(i)] $\varphi_{i,\varepsilon} \in \mathrm{Lip}(\R^N);$
  \item[(ii)] $\operatorname{supp}(\varphi_{i, \varepsilon}) \subseteq \Omega_\lambda$ and 
  $\varphi_{i, \varepsilon} \equiv 0$ near $\Gamma_\lambda$.
 \end{itemize}
 Moreover, since we know that $u_i < u_{i,\lambda}$ on $\mathcal{K}$, we also have
 \begin{equation} \label{eq:wilambdaequivzeroK}
  w_{i,\lambda}^+ =(u_i-u_{i,\lambda})^+\equiv 0\quad\text{on $\mathcal{K}$}.
  \end{equation}
 As a consequence, a standard density argument
 allows us to use
  $\varphi_{i,\ep}$ as a test function \emph{both in \eqref{eq:weaksoldef} and
  \eqref{eq.PDEulambda}}, from \eqref{eq:wilambdaequivzeroK} we obtain
  \begin{equation*}
	\begin{split}
	&
     \int_{\Omega_\lambda\setminus\mathcal{K}}
     \langle |\nabla u_i|^{p_i-2} \nabla u_i - |\nabla u_{i,\lambda}|^{p_i-2}\nabla u_{i,\lambda},\nabla
     \varphi_{i,\ep}\rangle\, dx \\
     &
     \qquad\qquad + \int_{\Omega_\lambda\setminus\mathcal{K}}
     \big(a_i(u_i)|\nabla u_i|^{q_i} - a_i(u_{i,\lambda})
     |\nabla u_{i,\lambda}|^{q_i}\big)\varphi_{i,\ep}\, dx \\
     & \quad
     = \int_{\Omega_\lambda\setminus\mathcal{K}}(f_i(\mathbf{u})-f_i(\mathbf{u}_\lambda))
     \varphi_{i,\ep}\, dx.
     \end{split}
  \end{equation*}
  Starting from this identity,
  and proceeding \emph{exactly} as in the proof of Lemma \ref{leaiuto} - \textbf{Case (i)}, 
  we obtain the following estimate, which is the analog of
  \eqref{eq:tousePoincar}:
  \begin{equation} \label{eq:tousePoincarThmCaseI}
  \int_{\Omega_\lambda\setminus\mathcal{K}}\big(|\nabla u_i| 
   + |\nabla u_{i,\lambda}|\big)^{p_i -2} |\nabla w_{i,\lambda}^{+}|^2\,dx 
   \leq C\sum_{j = 1}^m\int_{\Omega_\lambda\setminus\mathcal{K}}|w_{j,\lambda}^+|^2\,d x
  \end{equation}
  (here, $C > 0$ is a suitable constant depending on $p_i, q_i, a_i, f_i, \lambda,\Omega,M$ (with  $M = M_{\mathbf{u}} := \max_{1\leq j\leq m}
  \big(\|u_j\|_{L^\infty(\overline{\Omega}_{\rho+\tau_0})}+
  \|\nabla u_j\|_{L^\infty(\overline{\Omega}_{\rho+\tau_0})}\big)
  < +\infty$)).
  But, in this case,  \eqref{eq:tousePoincarThmCaseI} coincides with   \eqref{eq:integralestimatesperPoincare}.
 \medskip
 
 \noindent \textbf{Case (ii).}
 In this case, for every $\varepsilon \in (0,\e_0)$ we consider
  the maps
 \begin{align*}
 & \varphi_{i,\varepsilon}^{(1)}(x):= 
  \begin{cases}
  	e^{-A_i(u_i(x))}w_{i,\lambda}^+(x)\,\psi_\varepsilon^{p_i}(x),
  	& \text{if $x\in\Omega_\lambda$}, \\
  	0, & \text{otherwise},
  \end{cases}, \\[0.1cm]
 & \varphi_{i,\varepsilon}^{(2)}(x):= 
\begin{cases}
e^{-A_i(u_{i,\lambda}(x))}w_{i,\lambda}^+(x)\,\psi_\varepsilon^{p_i}(x),
& \text{if $x\in\Omega_\lambda$}, \\
0, & \text{otherwise},
\end{cases}
\end{align*}
where $A_i(t) := \int_0^t a_i(s)\,d s$. Also in the present case, we have already
recognized in the proof of Lemma \ref{leaiuto}  that
$\varphi_{i,\varepsilon}^{(1)},\,\varphi_{i,\varepsilon}^{(2)}$ satisfy the following
properties:
\begin{itemize}
	\item[(i)] $\varphi^{(1)}_{i,\varepsilon}, \varphi^{(2)}_{i,\varepsilon} \in \mathrm{Lip}(\R^N);$
	
	\item[(ii)] $\operatorname{supp}(\varphi^{(1)}_{i,\varepsilon}) \subseteq \Omega_\lambda$, $\operatorname{supp}(\varphi^{(2)}_{i,\varepsilon}) \subseteq \Omega_\lambda$ and $\varphi^{(1)}_{i, \varepsilon} \equiv \varphi^{(2)}_{i,\varepsilon} \equiv 0$ near $\Gamma_\lambda$.
\end{itemize}
As a consequence, 
a standard density argument shows that it is possible to use
$\varphi^{(1)}_{i,\ep},\,\varphi^{(2)}_{i,\ep}$ as test functions in \eqref{eq:weaksoldef} and \eqref{eq.PDEulambda}, respectively. By 
subtracting the resulting identities,
and by taking into account \eqref{eq:wilambdaequivzeroK}, we then obtain
\begin{equation*}
	\begin{split}
		&
		\int_{\Omega_\lambda\setminus\mathcal{K}}
		\langle |\nabla u_i|^{p_i-2} \nabla u_i , \nabla \varphi_{i, \varepsilon}^{(1)} \rangle \, dx -\int_{\Omega_\lambda\setminus\mathcal{K}}  \langle |\nabla u_{i,\lambda}|^{p_i-2}\nabla u_{i,\lambda},\nabla
		\varphi_{i,\ep}^{(2)}\rangle\, dx \\
		&
		\qquad\qquad + \int_{\Omega_{\lambda}\setminus\mathcal{K}} a_i(u_i)|\nabla u_i|^{p_i} \varphi_{i, \varepsilon}^{(1)} \, dx- \int_{\Omega_\lambda\setminus\mathcal{K}} a_i(u_{i,\lambda})
		|\nabla u_{i,\lambda}|^{p_i} \varphi_{i,\ep}^{(2)}\, dx \\
		& \quad
		= \int_{\Omega_\lambda\setminus\mathcal{K}} f_i(\mathbf{u}) \varphi_{i,\ep}^{(1)}\, dx- \int_{\Omega_\lambda\setminus\mathcal{K}} f_i(\mathbf{u}_\lambda) \varphi_{i,\ep}^{(2)}\, dx.
	\end{split}
\end{equation*}
Starting from this identity,
  and proceeding \emph{exactly} as in the proof of Lemma \ref{leaiuto} - \textbf{Case (ii)}, 
  we obtain the following estimate, which is again the analogue of
  \eqref{eq:tousePoincar}:
  \begin{equation} \label{eq:tousePoincarThmCaseII}
  \int_{\Omega_\lambda\setminus\mathcal{K}}\big(|\nabla u_i| 
   + |\nabla u_{i,\lambda}|\big)^{p_i -2} |\nabla w_{i,\lambda}^{+}|^2\,dx
   \leq C\sum_{j = 1}^m\int_{\Omega_\lambda\setminus\mathcal{K}}|w_{j,\lambda}^+|^2\,d x
  \end{equation}
  (here, $C > 0$ is a suitable constant depending on $p_i, q_i, a_i, f_i, \lambda,\Omega,M$ (with  $M = M_{\mathbf{u}} := \max_{1\leq j\leq m}
  \big(\|u_j\|_{L^\infty(\overline{\Omega}_{\rho+\tau_0})}+
  \|\nabla u_j\|_{L^\infty(\overline{\Omega}_{\rho+\tau_0})}\big)
  < +\infty$)).
  But it \eqref{eq:tousePoincarThmCaseII} coincides with
  \eqref{eq:integralestimatesperPoincare} in this case.
 \end{proof}

\end{document}